%% file: paper6d.tex
\documentclass[reqno]{amsart}
\usepackage{url}
\usepackage{amssymb}
\usepackage{graphicx}
\usepackage[colorinlistoftodos]{todonotes}
\usepackage{amsmath}
\usepackage{amsthm}

\newtheorem{theorem}{Theorem}[section]
\newtheorem{lemma}[theorem]{Lemma}
\newtheorem{corollary}[theorem]{Corollary}
\newtheorem{conjecture}[theorem]{Conjecture}

\theoremstyle{definition}
\newtheorem{definition}[theorem]{Definition}

\theoremstyle{remark}

\numberwithin{equation}{section}
\usepackage{algorithmic}
\usepackage[top=3cm,bottom=2cm,right=2cm,left=2cm]{geometry}

\begin{document}

\title[A Congruence Family For 2-Elongated Plane Partitions]{A Congruence Family For 2-Elongated Plane Partitions: An Application of the Localization Method}


\author{}
\address{}
\curraddr{}
\email{}
\thanks{}

\author{Nicolas Allen Smoot}
\address{}
\curraddr{}
\email{}
\thanks{}



\date{}

\dedicatory{}

\begin{abstract}
George Andrews and Peter Paule have recently conjectured an infinite family of congruences modulo powers of 3 for the 2-elongated plane partition function $d_2(n)$.  This congruence family appears difficult to prove by classical methods.  We prove a refined form of this conjecture by expressing the associated generating functions as elements of a ring of modular functions isomorphic to a localization of $\mathbb{Z}[X]$.
\end{abstract}

\maketitle

\section{Introduction}

\subsection{$k$-Elongated Plane Partitions}

In a recent paper \cite{AndrewsPaule}, George Andrews and Peter Paule study the properties of various plane partition functions.  In particular, they provide a large number of results for the $k$-elongated partition function $d_k(n)$, for $k\in\mathbb{Z}_{\ge 0}$.  This function is enumerated by

\begin{align}
D_k(q) := \sum_{n=0}^{\infty}d_k(n)q^n = \frac{(q^2;q^2)^k_{\infty}}{(q;q)^{3k+1}_{\infty}}\label{D2}.
\end{align}  For $n\ge 0$, $d_k(n)$ counts the number of unique directed graphs of $k$-elongated diamonds, shown in the diagrams below, in which each node is assigned a nonnegative integer label $a_j$, a directed edge $a_{b}\rightarrow a_c$ indicates that $a_b\ge a_c$, and the sum of the central linking nodes $a_1+a_{2k+2}+a_{4k+3}+...+a_{(2k+1)m+1}$ yields $n$.

\begin{figure}[h]
\centering
\scalebox{0.7}{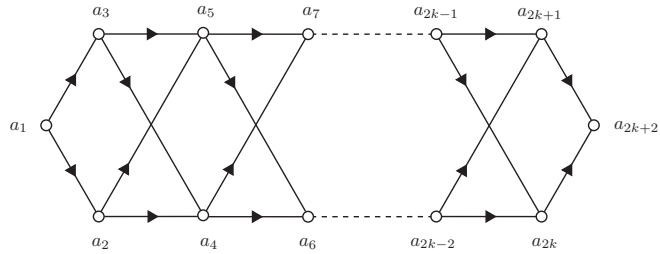}
\caption{A schematic of a $k$-elongated partition diamond of length 1.}\label{kelong1}
\end{figure}

\begin{figure}[h]
\centering
\scalebox{0.7}{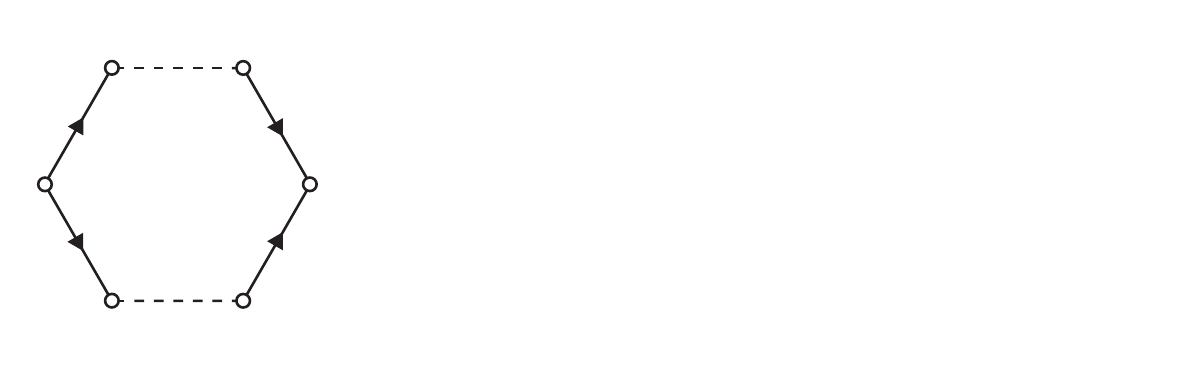}
\caption{A schematic of a $k$-elongated partition diamond of length $m$.}\label{kelong2}
\end{figure}

Notice that $d_k(n)$ generalizes the unrestricted partition function, since $d_0(n)=p(n)$.  Moreover, $k$-elongated plane partitions are also a generalization of partition diamonds, which were developed in \cite{AndrewsPaule2} as an application of the algorithmic procedures of Partition Analysis begun with MacMahon \cite[Vol. II, Section VIII]{MacMahon}, and fully developed by Andrews, Paule, and Riese in their series of papers on the subject.

One peculiar characteristic of $d_k(n)$ is a richness of arithmetic properties which is surprising even by the standards of partition-type functions.  Andrews and Paule discover many such properties in \cite[Section 7]{AndrewsPaule}.  In particular, they propose a conjecture \cite[Section 7, Conjecture 3]{AndrewsPaule} regarding an infinite family of congruences for $d_2(n)$:

\begin{conjecture}[G.E. Andrews, P. Paule]\label{conj}
Let $n,\alpha\in\mathbb{Z}_{\ge 1}$ such that $8n\equiv 1\pmod{3^{\alpha}}$.  Then $d_2(n)\equiv 0\pmod{3^{\alpha}}$.
\end{conjecture}

Andrews and Paule state that this conjecture ``seem[s] to be particularly challenging" \cite[Section 8]{AndrewsPaule}.  In this paper we refine and prove Conjecture \ref{conj}.

\subsection{A Brief History of Partition Congruence Families}

The conjecture above resembles the form of a very large number of similar partition congruence families that have been studied over the last century.  The first of such congruence families came from Ramanujan's groundbreaking work in 1918 \cite{Ramanujan}.  He conjectured that if $\ell\in\{5,7,11\}$ and $n,\alpha\ge 1$ are positive integers such that $24n\equiv 1\pmod{\ell^{\alpha}}$, then $p(n)\equiv 0\pmod{\ell^{\alpha}}.$  This conjecture has since been proved, with a slight modification in the case of $\ell=7$:

\begin{align}
\text{If } 24n\equiv 1\pmod{5^{\alpha}},&\text{ then } p(n)\equiv 0\pmod{5^{\alpha}},\label{origconjecR5}\\
\text{If } 24n\equiv 1\pmod{7^{\alpha}},&\text{ then } p(n)\equiv 0\pmod{7^{\left\lfloor\alpha/2\right\rfloor + 1}},\label{origconjecR7}\\
\text{If } 24n\equiv 1\pmod{11^{\alpha}},&\text{ then } p(n)\equiv 0\pmod{11^{\alpha}}\label{origconjecR11}.
\end{align}

Many similar congruence families have since been found, in a very large variety of more restrictive partition functions.  Such congruence families reveal an astonishing depth of arithmetic properties within integer partitions, which do not appear at first sight to have any substantial arithmetic structure.

More surprising still is the enormous variety in the difficulty of the proofs of such congruence families.  Many are proved routinely using classical methods.  For example, Watson proved (\ref{origconjecR5}) and (\ref{origconjecR7}) in 1938 \cite{Watson}, and Ramanujan himself appears to have understood the proof of (\ref{origconjecR5}) years earlier \cite{BO}.  Many congruence families can be proved with similar techniques used by Ramanujan and Watson.

Other examples are much more arduous.  The case of (\ref{origconjecR11}) was not proved until Atkin's work in 1967 \cite{Atkin}.  Even today some contemporary families can take years to resolve.  The Andrews--Sellers congruence family for generalized 2-colored Frobenius partitions was proposed in 1994, but it was not proved until 2012 \cite{Paule}.  Conjecture \ref{conj} appears at first sight to be of the more difficult variety, as Andrews and Paule observe \cite[Sections 7, 8]{AndrewsPaule}.

Nevertheless, the opening steps of a proof of a given infinite family of partition congruences are almost always the same: we identify an associated modular curve (usually $\mathrm{X}_0(N)$ for some $N\ge 2$), and construct a sequence of modular functions $\mathcal{L}:=\left( L_{\alpha} \right)_{\alpha\ge 1}$ which enumerate the cases of the congruence family for each power $\alpha$.  In the case of Conjecture \ref{conj}, we have the associated modular curve $\mathrm{X}_0(6)$, and the functions $L_{\alpha}$ have the form

\begin{align*}
L_{\alpha} = \Phi_{\alpha}(q)\cdot \sum_{8n\equiv 1\bmod{3^{\alpha}}} d_2(n)q^{\left\lfloor n/3^{\alpha} \right\rfloor+1},
\end{align*} with $q=e^{2\pi i\tau}$, $\tau\in\mathbb{H}$ and $\Phi_{\alpha}(q)-1\in q\mathbb{Z}[[q]]$ (see Section 2).

One then defines a sequence of linear operators (generally an alternating sequence) $U^{(\alpha)}$, which are usually modified forms of the classical operator $U_{\ell}$ for the relevant prime $\ell$ (in our case, $U_3$), and in which $U^{(\alpha)}(L_{\alpha})=L_{\alpha+1}$.

We want to know how to properly represent $L_{\alpha}$ in terms of some useful reference functions, and how such functions behave on the application of $U^{(\alpha)}$.  Because the reference functions most useful to us are themselves modular functions, they are generally subject to a modular equation.  From this we can construct a recurrence relation for the application of $U^{(\alpha)}$ on the reference functions, and ultimately an induction proof for the overall congruence family.

This approach is common to all known proofs of infinite congruence families.  Indeed, it was understood at least roughly by Ramanujan and Watson.  However, in the cases (\ref{origconjecR5}) and (\ref{origconjecR7}) that Ramanujan and Watson proved, the topology of the associated modular curves is very simple.  In particular, the genus $\mathfrak{g}$ of each curve is 0, and the cusp count $\epsilon_{\infty}$ of each curve is 2.  This reduces the associated function spaces to single-variable polynomial rings.  Thus, a detailed understanding of the underlying modular curve is not necessary in these cases.

Difficulties quickly mount once $\mathfrak{g}$ or $\epsilon_{\infty}$ grow even slightly.  Lehner \cite{Lehner6}, \cite{Lehner7}, and later Atkin \cite{Atkin} understood that the primary difficulty in the proof of (\ref{origconjecR11}) is that the genus of the associated modular curve is 1.  The Weierstrass gap theorem \cite{Paule2} thus requires a more intricate algebraic structure to represent the analogues to the functions $L_{\alpha}$.

The cusp count further complicates matters.  If $\epsilon_{\infty}\left( \mathrm{X}_0(N) \right)=2$---as in the case for $\mathrm{X}_0(\ell)$ with $\ell$ a prime---then the functions $L_{\alpha}$ would necessarily have a zero at one cusp and a pole at the other \cite[Chapter V, Theorem 4F]{Lehner}.  This would ensure that $L_{\alpha}$ will have a polynomial representation in terms of a given set of reference functions.  It is this representation which would likely hold the arithmetic information of interest to us.  This kind of representation is indeed common to (\ref{origconjecR5}), (\ref{origconjecR7}), and even (\ref{origconjecR11}), since the associated modular curve for each of these families has a prime level.

In the case of Conjecture \ref{conj}, we have $\mathfrak{g}\left(\mathrm{X}_0(6)\right)=0$ and $\epsilon_{\infty}\left(\mathrm{X}_0(6)\right)=4.$  Thus the genus simplifies the problem, while the cusp count complicates it.  The zero genus means that our space of modular functions meromorphic at a single cusp of the corresponding modular curve will be as simple as in the cases of (\ref{origconjecR5}) and (\ref{origconjecR7}).  In particular, the space is isomorphic to $\mathbb{C}[x]$ for a Hauptmodul $x$.

However, due to the larger number of cusps, $L_{\alpha}$ may have poles at multiple points, and a polynomial expression in terms of a Hauptmodul is far from assured.  We therefore need to express each $L_{\alpha}$ as a member of a certain space of functions meromorphic at multiple cusps.  Importantly, this space must have stability under repeated application of $U^{(\alpha)}$.  Atkin appears to have been the first to recognize the importance of organizing modular functions by this property \cite{Atkin}.  However, Atkin worked over $\mathrm{X}_0(11)$, with only two cusps, and thus function spaces reducible to polynomial representations for the relevant functions.  We need to generalize to a larger number of cusps.

This is a new approach to the problem of congruence families, which to our knowledge has yet to be carefully examined.  The conceptual linchpin is that we can push each $L_{\alpha}$ into the space of modular functions with a pole at a single cusp of $\mathrm{X}_0(6)$.  This can be done by multiplying each $L_{\alpha}$ by a suitable nonnegative power $n_{\alpha}$ of a specific modular function $z$ which has a pole at a single cusp, and positive order at all of the other cusps for which $L_{\alpha}$ has poles, as well as an integer power series expansion in the Fourier variable $q$, and leading coefficient 1.  So long as we are working over classical modular curves, such a modular function is guaranteed to exist \cite[Lemma~20]{Radu}, \cite{Newman}.

Because $z^{n_{\alpha}}$ and $z^{n_{\alpha}} L_{\alpha}$ are both within the space of modular functions with a pole at a single cusp, they may both be expressed in polynomial form in terms of a given Hauptmodul $x$.  Upon division by $z^{n_{\alpha}}$, we have a representation of $L_{\alpha}$ not as a polynomial, but rather as a \textit{rational polynomial in} $x$.  If $x,z$ are chosen well, then the coefficients of the numerator may contain the arithmetic information of interest to us.  In particular, because $z$ has leading coefficient 1, the presence of $z$ in the denominator will not impede our efforts to determine divisibility properties which might be glimpsed from the numerator.

As we will see, this allows us to study a large class of modular functions (localized polynomials in a Hauptmodul), with poles at multiple cusps.  

\subsection{Main Theorem}

Using the approach discussed above, we have proved the following:

\begin{theorem}\label{Thm12}
Let $n,\alpha\in\mathbb{Z}_{\ge 1}$ such that $8n\equiv 1\pmod{3^{\alpha}}$.  Then $d_2(n)\equiv 0\pmod{3^{2\left\lfloor\alpha/2\right\rfloor+1}}$.
\end{theorem}

\noindent Notice that this theorem extends Conjecture \ref{conj} by accounting for an additional power of 3 when $\alpha$ is even.

As an example, we take the first case of Theorem \ref{Thm12}.  The first function of $\mathcal{L}$  has the form

\begin{align}
L_{1} &= \frac{(q^3;q^3)^7_{\infty}}{(q^6;q^6)^2_{\infty}} \sum_{n=0}^{\infty}d_2\left(3n + 2\right)q^{n+1}.\label{L1def}
\end{align}  With the Hauptmodul

\begin{align}
x = x(\tau) :=& q\frac{(q^2;q^2)_{\infty}(q^6;q^6)_{\infty}^5}{(q;q)_{\infty}^5(q^3;q^3)_{\infty}},\label{xdef}
\end{align} we have the following identity:

\begin{align}
L_1 =& \frac{1}{1+9x}\cdot \left( 33x+1392x^2+21120x^3+138240x^4+331776x^5 \right).\label{L1S}
\end{align}  Given that $1+9x\equiv 1\pmod{3}$, the divisibility of the numerator coefficients by 3 yields the information we need.

Such a representation exists for every $L_{\alpha}\in\mathcal{L}$.  Indeed, Theorem \ref{Thm12} and Conjecture \ref{conj} immediately arise from the following theorem---the central result of this paper:

\begin{theorem}\label{Mythm}
Let

\begin{align*}
\psi:=\left\lfloor \frac{3^{\alpha+1}}{8} \right\rfloor,\ \ \ \ \ \beta := 2\left\lfloor\alpha/2\right\rfloor+1.
\end{align*}  For all $\alpha\ge 1$,

\begin{align}
\frac{(1+9x)^{\psi}}{3^{\beta}}\cdot L_{\alpha}\in\mathbb{Z}[x].\label{myeqn}
\end{align}

\end{theorem}

The methods just described were used to prove Theorem \ref{Mythm}, and therefore Theorem \ref{Thm12} and Conjecture \ref{conj}, with astonishing rapidity.  The author successfully constructed the proof given in this article less than a week after first glimpsing Conjecture \ref{conj}.  We stress this point to emphasize the utility of our localization method for resolving congruence conjectures.  We also point out that this technique has already been applied \cite{Smoot0} to prove a congruence family associated with a family of weakly modular forms.  Indeed, many of the key steps in the proof given here are similar to the steps specified in the proof of the congruence family in \cite{Smoot0} (though in our current case, the question of proving the associated initial relations is easier than the analogous relations in \cite{Smoot0}).  We are compelled to believe that these techniques will prove useful elsewhere in the theory of modular forms and its applications.  We hope that other researchers will benefit from our work.

The remainder of our paper is organized as follows:  In Section \ref{basictheorysection} we will briefly review the necessary theory of modular functions.  In Section \ref{proofsetup} we will more precisely define the modular functions $L_{\alpha}$, together with an associated alternating linear operating sequence $U^{(i)}$, with $i\equiv\alpha\pmod{2}$.  In Section \ref{TheModularEquations} we also define our functions $x,z$, together with their modular equations.

In Section \ref{algebrastructure} we will define the localized polynomial ring that is most important to us, together with some important subgroups $\mathcal{V}_n^{(i)}$.  We will also use the modular equations for $x,z$ to develop recurrences for $U^{(i)}\left( x^m/(1+9x)^n \right)$.

In Section \ref{maintheoremsection} we will use these recurrences to study how the operator $U^{(i)}$ takes elements from $\mathcal{V}_n^{(i)}$ and sends them to $\mathcal{V}_{n'}^{(i+1)}$ for some $n'>n$.  In particular, we will show that elements from $\mathcal{V}_n^{(1)}$ will gain two powers of 3 upon being sent to $\mathcal{V}_{3n}^{(0)}$.  We then demonstrate that $L_1\in \mathcal{V}_1^{(1)}$, and complete the proof of Theorem \ref{Mythm}.

Section \ref{initialrelationssection} will show how to prove the 18 initial relations necessary for the recurrences of Section \ref{algebrastructure}.  We show that these relations are algebraically dependent on 6 fundamental relations: namely, $U^{(i)}\left( x^m\right)$ for $i=0,1$, and $0\le m\le 2$.  These relations can easily be proved, along with some additional properties involving the functions $x$ and $z$, using the finiteness principles of the modular cusp analysis.

Appendix I will cover some additional computational details needed in Section \ref{maintheoremsection}.  Appendix II will list the 6 fundamental relations from Section \ref{initialrelationssection}.  We also provide a Mathematica supplement which will detail many of the other calculations too cumbersome to include here.  All of the key computations in this paper can be justified by the algorithms developed by Silviu Radu in \cite{Radu}, and with a software implementation in \cite{Smoot1}.  However, for the sake of accessibility, we include all of the necessary computational procedures in our supplement, without requiring any additional software beyond Version 12.2 of Wolfram Mathematica.  The supplement can be found online at \url{https://www.risc.jku.at/people/nsmoot/2eplanepartsupplement.nb}.

\section{Basic Theory}\label{basictheorysection}

We preface our results with an overview of the theory of modular functions.  Given the importance of the topology of the associated Riemann surface to constructing a proof of Theorem \ref{Mythm}, we will (briefly!) define the modular curve $\mathrm{X}_0(N)$ and highlight some of the key topological numbers mentioned in the Introduction.  We will also define a modular function for our purposes, some key finiteness conditions imposed by the theory, and some useful tools for computing the order of certain modular functions.  We will also define the $U_3$ operator and describe some of its properties.

Our treatment is necessarily limited for want of space.  For a classic review of the theory, see \cite{Knopp}.  For a more modern treatment, see \cite{Diamond}.

\subsection{Preliminaries}

Define $\mathbb{H}$ to be the upper half complex plane, and $\hat{\mathbb{H}}:=\mathbb{H}\cup\{\infty\}\cup\mathbb{Q}$.  We define $\hat{\mathbb{Q}}:=\mathbb{Q}\cup\{\infty\}$, and for any $a\neq 0$ we let  $a/0=\infty$.  We also define $\mathrm{SL}(2,\mathbb{Z})$ as the set of all $2\times 2$ integer matrices with determinant~1.

For any given $N\in\mathbb{Z}_{\ge 1}$, let

\begin{align*}
\Gamma_0(N) = \Bigg\{ \begin{pmatrix}
  a & b \\
  c & d 
 \end{pmatrix}\in \mathrm{SL}(2,\mathbb{Z}) : N|c \Bigg\}.
\end{align*}

We define a group action

\begin{align*}
\Gamma_0(N)\times\hat{\mathbb{H}}&\longrightarrow\hat{\mathbb{H}},\\
\left(\begin{pmatrix}
  a & b \\
  c & d 
 \end{pmatrix},\tau\right)&\longrightarrow \frac{a\tau+b}{c\tau+d}.
\end{align*} If $\gamma=\begin{pmatrix}
  a & b \\
  c & d 
 \end{pmatrix}$ and $\tau\in\hat{\mathbb{H}}$, then we write
 
 \begin{align*}
 \gamma\tau := \frac{a\tau+b}{c\tau+d}.
 \end{align*}  The orbits of this action are defined as

\begin{align*}
[\tau]_N := \left\{ \gamma\tau: \gamma\in\Gamma_0(N) \right\}.
\end{align*}

\begin{definition}
For any $N\in\mathbb{Z}_{\ge 1}$, we define the classical modular curve of level $N$ as the set of all orbits of $\Gamma_0(N)$ applied to $\hat{\mathbb{H}}$:

\begin{align*}
\mathrm{X}_0(N):=\left\{ [\tau]_N : \tau\in\hat{\mathbb{H}} \right\}
\end{align*}
\end{definition}

We can restrict the group action to $\hat{\mathbb{Q}}$: for each $\tau\in\hat{\mathbb{Q}}$, $[\tau]_N\subseteq\hat{\mathbb{Q}}$.  There are only a finite number of such orbits \cite[Section 3.8]{Diamond}.

\begin{definition}
For any $N\in\mathbb{Z}_{\ge 1}$, the cusps of $\mathrm{X}_0(N)$ are the orbits of $\Gamma_0(N)$ applied to $\hat{\mathbb{Q}}$.
\end{definition}

Two rational elements $a_1/c_1, a_2/c_2$ are in the same cusp if and only if \cite[Proposition 3.8.3]{Diamond} there exist some $j,y\in\mathbb{Z}$ with $\mathrm{gcd}(y,N)=1$ and

\begin{align}
y\cdot a_2&\equiv a_1+j\cdot c_1\pmod{N},\label{equivalentcuspsA}\\
c_2&\equiv y\cdot c_1\pmod{N}.\label{equivalentcuspsB}
\end{align}

A detailed description of the Riemann surface structure of $\mathrm{X}_0(N)$ is given in \cite[Chapters 2,3]{Diamond}.  We will additionally state that $\mathrm{X}_0(N)$ possesses two unique nonnegative integers of interest to us: the genus, denoted $\mathfrak{g}\left(  \mathrm{X}_0(N) \right)$, and the number of cusps, denoted $\epsilon_{\infty}\left(  \mathrm{X}_0(N) \right)$.  These are both referenced in the Introduction, and can be computed according to formulae in \cite[Chapter 3, Sections 3.1, 3.8]{Diamond}.

\begin{definition}\label{DefnModular}
Let $f:\mathbb{H}\longrightarrow\mathbb{C}$ be holomorphic on $\mathbb{H}$.  Then $f$ is a modular function over $\Gamma_0(N)$ if the following properties are satisfied for every $\gamma=\left(\begin{smallmatrix}
  a & b \\
  c & d 
 \end{smallmatrix}\right)\in\mathrm{SL}(2,\mathbb{Z})$:

\begin{enumerate}
\item If $\gamma\in\Gamma_0(N)$, we have $\displaystyle{f\left( \gamma\tau \right) = f(\tau)}.$
\item We have $$\displaystyle{f\left( \gamma\tau \right) = \sum_{n=n_{\gamma}}^{\infty}\alpha_{\gamma}(n)q^{n\gcd(c^2,N)/ N}},$$  with $n_{\gamma}\in\mathbb{Z}$, and $\alpha_{\gamma}(n_{\gamma})\neq 0$.  If $n_{\gamma}\ge 0$, then $f$ is holomorphic at the cusp $[a/c]_N$.  Otherwise, $f$ has a pole of order $n_{\gamma}$, and principal part
 \begin{align}
 \sum_{n=n_{\gamma}}^{-1}\alpha_{\gamma}(n)q^{n\gcd(c^2,N)/ N},\label{princpartmod}
 \end{align} at the cusp $[a/c]_N$.
 \end{enumerate}  We refer to $\mathrm{ord}_{a/c}^{(N)}(f) := n_{\gamma}(f)$ as the order of $f$ at the cusp $[a/c]_N$.
\end{definition}

We now define the relevant sets of all modular functions:

\begin{definition}
Let $\mathcal{M}\left(\Gamma_0(N)\right)$ be the set of all modular functions over $\Gamma_0(N)$, and $\mathcal{M}^{a/c}\left(\Gamma_0(N)\right)\subset \mathcal{M}\left(\Gamma_0(N)\right)$ to be those modular functions over $\Gamma_0(N)$ with a pole only at the cusp $[a/c]_N$.  These are commutative algebras with 1, and standard addition and multiplication \cite[Section 2.1]{Radu}.
\end{definition}

Due to its precise symmetry over $\Gamma_0(N)$, any modular function $f\in\mathcal{M}\left( \Gamma_0(N) \right)$ induces a well-defined function

\begin{align*}
\hat{f}&:\mathrm{X}_0(N)\longrightarrow\mathbb{C}\cup\{\infty\}\\
&:[\tau]_N\longrightarrow f(\tau).
\end{align*}  The notions of pole order and cusps of $f$ used in Definition \ref{DefnModular} are constructed so as to coincide with these notions applied to $\hat{f}$ on $\mathrm{X}_0(N)$.  Indeed, the set of meromorphic functions on $\mathrm{X}_0(N)$ with poles only at the cusps has a one-to-one correspondence with $\mathcal{M}\left(\Gamma_0(N)\right)$ \cite[Chapter VI, Theorem 4A]{Lehner}.  In particular, (\ref{princpartmod}) represents the principal part of $\hat{f}$ in local coordinates near the cusp $[a/c]_N$.  Notice that as $\tau\rightarrow i\infty$, we must have $\gamma\tau\rightarrow a/c$, and $q\rightarrow 0$.

Because $f$ is holomorphic on $\mathbb{H}$ by definition, all possible poles for $\hat{f}$ must be found for $[\tau]_N\subseteq\hat{\mathbb{Q}}$.

\begin{theorem}\label{riemannsurfacetheorema}
Let $\mathrm{X}$ be a compact Riemann surface, and let $\hat{f}:\mathrm{X}\longrightarrow\mathbb{C}$ be analytic on all of $\mathrm{X}$.  Then $\hat{f}$ must be a constant function.
\end{theorem}

This theorem is central to our work, as can be shown in the following corollary:

\begin{corollary}
For a given $N\in\mathbb{Z}_{\ge 1}$, if $f\in\mathcal{M}\left(\Gamma_0(N)\right)$ has positive order at every cusp of $\Gamma_0(N)$, then $f$ must be a constant.
\end{corollary}

To illustrate the importance of these results, we consider $f,g\in\mathcal{M}^{\infty}\left(\Gamma_0(N)\right)$.

Suppose that the principal parts of each function match.  In that case, $f-g\in\mathcal{M}\left(\Gamma_0(N)\right)$ will have no poles whatsoever, and will thus be analytic on the whole of $\mathrm{X}_0(N)$.  Theorem \ref{riemannsurfacetheorema} forces $\hat{f}-\hat{g}$, and therefore $f-g$, to be a constant.  Of course, we can quickly compare the constant terms of both functions.

We now want to show how to construct modular functions.  Most useful to us in this case is Dedekind's eta function \cite[Chapter 3]{Knopp}:

\begin{align*}
\eta(\tau):= e^{\pi i\tau/12}\prod_{n=1}^{\infty}\left( 1-e^{2\pi i n\tau} \right).
\end{align*}

The following theorem is due to Newman \cite[Theorem 1]{Newman}:

\begin{theorem}\label{Newmanntheorem}
Let $f = \prod_{\delta | N} \eta(\delta\tau)^{r_{\delta}}$, with $\hat{r} = (r_{\delta})_{\delta | N}$ an integer-valued vector, for some $N\in\mathbb{Z}_{\ge 1}$.  Then $f\in\mathcal{M}\left(\Gamma_0(N)\right)$ if and only if the following apply:

\begin{enumerate}
\item $\sum_{\delta | N} r_{\delta} = 0;$
\item $\sum_{\delta | N} \delta r_{\delta} \equiv 0\pmod{24};$
\item $\sum_{\delta | N} \frac{N}{\delta}r_{\delta} \equiv 0\pmod{24};$
\item $\prod_{\delta | N} \delta^{|r_{\delta}|}$ is a perfect square.
\end{enumerate}

\end{theorem}

We now want to be able to examine the order of a given eta quotient at some cusp of the corresponding modular curve.  The following theorem can be found in \cite[Theorem 23]{Radu}, though it is attributed to Ligozat:

\begin{theorem}\label{Ligozat}
If $f = \prod_{\delta | N} \eta(\delta\tau)^{r_{\delta}}\in\mathcal{M}\left(\Gamma_0(N)\right)$, then the order of $f$ at the cusp $[a/c]_N$ is given by the following:

\begin{align*}
\mathrm{ord}_{a/c}^{(N)}(f) = \frac{N}{24\gcd{(c^2,N)}}\sum_{\delta | N} r_{\delta}\frac{\gcd{(c,\delta)}^2}{\delta}.
\end{align*}

\end{theorem}

Finally, we give a theorem from \cite[Theorem 39]{Radu} which we will use to identify our function $z$.

\begin{theorem}\label{orderboundmodfunc}
Suppose that for some $M\in\mathbb{Z}_{\ge 1}$ and an index-valued vector $(r_{\delta})_{\delta | M}$ we define the function $a(n)$ by
\begin{align*}
\sum_{n=1}^{\infty} a(n)q^n := \prod_{\delta | M} (q^{\delta};q^{\delta})_{\infty}^{r_{\delta}}
\end{align*}  Moreover, suppose that $m,t,N\in\mathbb{Z}_{\ge 0}$ such that $0\le t< m$, and that

\begin{align*}
f = \prod_{\lambda | N} \eta(\lambda\tau)^{s_{\delta}} \sum_{n=1}^{\infty} a(mn+t)q^n\in\mathcal{M}\left( \Gamma_0(N) \right).
\end{align*}  Then

\begin{align*}
\mathrm{ord}_{a/c}^{(N)}(f) \ge \frac{N}{\mathrm{gcd}(c^2,N)} \left(\min\limits_{0\le l\le m-1} \frac{1}{24}\sum_{\delta | M} r_{\delta} \frac{\mathrm{gcd}(\delta (a+l\cdot c\cdot \mathrm{gcd}(m^2-1,24)),mc)^2}{\delta m} + \frac{1}{24}\sum_{\lambda | N} \frac{s_{\lambda}\mathrm{gcd}(\lambda, c)^2}{\lambda}\right).
\end{align*}

\end{theorem}

Finally, let us properly define $U_3$ and its effect on $\mathcal{M}\left( \Gamma_0(N) \right)$:

\begin{definition}
Let $f(q) = \sum_{m\ge M}a(m)q^m$.  Then
\begin{align}
U_3\left(f(q)\right) := \sum_{3m\ge M} a(3m)q^m.\label{defU3}
\end{align}
\end{definition}

We list some of the important properties of $U_3$.  The proofs can be found in \cite[Chapter 8]{Knopp}.

\begin{lemma}\label{impUprop}
Given two functions 
\begin{align*}
f(q) = \sum_{m\ge M}a(m)q^m,\ g(q) = \sum_{m\ge N}b(m)q^m,
\end{align*} any $\alpha\in\mathbb{C}$, a primitive third root of unity $\zeta$, and the convention that $q^{1/3}~:=~e^{2\pi i\tau/3}$, we have the following:
\begin{enumerate}
\item $U_3\left(\alpha\cdot f+g\right) = \alpha\cdot U_3\left(f\right) + U_3\left(g\right)$;
\item $U_3\left(f(q^3)g(q)\right) = f(q) U_3\left(g(q)\right)$;
\item $3\cdot U_3\left(f\right) = \sum_{r=0}^2 f\left( \zeta^rq^{1/3} \right)$.
\end{enumerate}
\end{lemma}

As a final property of $U_3$, we give an important result from \cite[Lemma 7]{AtkinL}, which we will use in Section \ref{initialrelationssection}:

\begin{theorem}\label{uelleffectsmod}
For any $N\in\mathbb{Z}_{\ge 1}$ with $3 | N$ and $f\in\mathcal{M}\left( \Gamma_0(N) \right)$,
\begin{align*}
U_{3}\left(f \right)&\in\mathcal{M}\left( \Gamma_0(N) \right).
\end{align*}  Moreover, if $3^2 | N$, then
\begin{align*}
U_{3}\left(f \right)&\in\mathcal{M}\left( \Gamma_0(N/3) \right).
\end{align*}
\end{theorem}

\subsection{Proof Setup}\label{proofsetup}

The initial steps of our proof are essentially identical to those of the classical theory.  We need to define our functions $L_{\alpha}$, together with the operators $U^{(i)}$.  Hereafter, we denote $q:=e^{2\pi i\tau}$, for $\tau\in\mathbb{H}$.  We begin by defining an important auxiliary function:

\begin{align}
\mathcal{A} &:= q\frac{D_2(q)}{D_2(q^9)} = q\frac{(q^2;q^2)^2_{\infty}(q^9;q^9)^7_{\infty}}{(q;q)^7_{\infty}(q^{18};q^{18})^2_{\infty}}.\label{Zdef}
\end{align}  We can use Theorem \ref{Newmanntheorem} to verify that $\mathcal{A}(\tau)\in\mathcal{M}\left( \Gamma_0(18) \right)$.

The appeal of $\mathcal{A}$, besides its modularity, is that applying the classical $U_3$ operator onto it will produce the generating function $L_1$ that we need.

We will define the operators $U^{(i)}$ by 

\begin{align}
U^{(1)}(f) :=& U_3(f),\\
U^{(0)}(f) :=& U_3(\mathcal{A}\cdot f),\\
U^{(\alpha)}(f) :=& U^{(i)}(f),\ \alpha\equiv i\pmod{2}.
\end{align}

Our main generating functions $L_{\alpha}$ for each case of Theorem \ref{Thm12} are defined as follows:
\begin{align}
L_0 &:= 1,\label{L0d}\\
L_{2\alpha-1}(\tau) &= \frac{(q^3;q^3)^7_{\infty}}{(q^6;q^6)^2_{\infty}}\cdot\sum_{n=0}^{\infty} d_2(3^{2\alpha-1}n + \lambda_{2\alpha-1})q^{n+1},\label{Lod}\\
L_{2\alpha}(\tau) &= \frac{(q;q)^7_{\infty}}{(q^2;q^2)^2_{\infty}}\cdot\sum_{n=0}^{\infty} d_2(3^{2\alpha}n + \lambda_{2\alpha})q^{n+1},\label{Led}
\end{align} with the $\lambda_{\alpha}$ defined as

\begin{align}
\lambda_{2\alpha-1} &:= \frac{1+5\cdot 3^{2\alpha-1}}{8},\label{lamodef}\\
\lambda_{2\alpha} &:= \frac{1+7\cdot 3^{2\alpha}}{8}\label{lamedef}.
\end{align}  In either case, $\lambda_{\alpha}\in\mathbb{Z}$ are the minimal positive solutions to 

\begin{align*}
8x\equiv 1\pmod{3^{\alpha}}.
\end{align*}  Therefore, one could write $L_{\alpha}$ in the form

\begin{align*}
L_{\alpha} = \Phi_{\alpha}(q)\cdot \sum_{8n\equiv 1\bmod{3^{\alpha}}} d_2(n)q^{\left\lfloor n/3^{\alpha} \right\rfloor+1}.
\end{align*}  Notice that

\begin{align}
U_3(\mathcal{A}) &= U_3\left( q\frac{D_2(q)}{D_2(q^9)} \right) = \frac{1}{D_2(q^3)}\cdot U_3\left( q D_2(q) \right) = \frac{1}{D_2(q^3)}\cdot U_3\left( \sum_{n=1}^{\infty}d_2(n-1)q^n \right)\label{isittruethough}\\
&=\frac{1}{D_2(q^3)}\cdot \sum_{n=1}^{\infty}d_2(3n-1)q^n =\frac{1}{D_2(q^3)}\cdot \sum_{n=0}^{\infty}d_2(3n+2)q^{n+1}\label{isittruethougha}\\
&= L_1.\label{isittruethoughb}
\end{align}  More generally, the $U_3$ operator provides us with a convenient means of accessing $L_{\alpha+1}$ from $L_{\alpha}$, as the following lemma shows:

\begin{lemma}
For all $\alpha\ge 0$, we have
\begin{align}
L_{\alpha+1} &= U^{(\alpha)}\left( L_{\alpha} \right).\label{Lete}
\end{align}
\end{lemma}

\begin{proof}
For any $\alpha\ge 1$,

\begin{align*}
U^{(2\alpha-1)}\left( L_{2\alpha-1} \right)=U_3\left( L_{2\alpha-1} \right) &= U_3\left( \frac{1}{D_2(q^3)} \sum_{n\ge 0} d_2\left(3^{2\alpha-1}n + \lambda_{2\alpha-1}\right)q^{n+1} \right)\\
&= \frac{1}{D_2(q)}\cdot U_3\left( \sum_{n\ge 1}d_2\left(3^{2\alpha-1}(n-1) + \lambda_{2\alpha-1}\right)q^{n} \right)\\
&= \frac{1}{D_2(q)}\cdot \sum_{3n\ge 1}d_2\left(3^{2\alpha-1}(3n-1) + \lambda_{2\alpha-1}\right)q^{n}\\
&= \frac{1}{D_2(q)}\cdot \sum_{n\ge 1}d_2\left(3^{2\alpha}n-3^{2\alpha-1} + \lambda_{2\alpha-1}\right)q^{n}\\
&= \frac{1}{D_2(q)}\cdot \sum_{n\ge 0}d_2\left(3^{2\alpha}n+3^{2\alpha}-3^{2\alpha-1} + \lambda_{2\alpha-1}\right)q^{n+1}\\
&= \frac{1}{D_2(q)}\cdot \sum_{n\ge 0}d_2\left(3^{2\alpha}n+\lambda_{2\alpha}\right)q^{n+1}.
\end{align*}  Similarly,

\begin{align*}
U^{(2\alpha)}\left( L_{2\alpha} \right)=U_3\left( \mathcal{A}\cdot L_{2\alpha} \right) &= U_3\left( q\frac{D_2(q)}{D_2(q^9)}\frac{1}{D_2(q)} \sum_{n\ge 0} d_2\left(3^{2\alpha}n + \lambda_{2\alpha}\right)q^{n+1} \right)\\
&= \frac{1}{D_2(q^3)}\cdot U_3\left( \sum_{n\ge 2}d_2\left(3^{2\alpha}(n-2) + \lambda_{2\alpha}\right)q^{n} \right)\\
&= \frac{1}{D_2(q^3)}\cdot \sum_{3n\ge 3}d_2\left(3^{2\alpha}(3n-2) + \lambda_{2\alpha}\right)q^{n}\\
&= \frac{1}{D_2(q^3)}\cdot \sum_{n\ge 1}d_2\left(3^{2\alpha+1}n-2\cdot 3^{2\alpha} + \lambda_{2\alpha}\right)q^{n}\\
&= \frac{1}{D_2(q^3)}\cdot \sum_{n\ge 0}d_2\left(3^{2\alpha+1}(n+1)-2\cdot 3^{2\alpha} + \lambda_{2\alpha}\right)q^{n+1}\\
&= \frac{1}{D_2(q^3)}\cdot \sum_{n\ge 0}d_2\left(3^{2\alpha+1}n+\lambda_{2\alpha+1}\right)q^{n+1}.
\end{align*}
\end{proof}

Our goal henceforth will be to understand the functions $L_{\alpha}$

\subsection{Our Hauptmodul}\label{TheModularEquations}

Due to the fact that $\mathcal{A}$ is modular over $\Gamma_0(18)$, the functions $L_{\alpha}$ are modular over $\Gamma_0(6)$.  As described in the Introduction, we want to identify two key functions: one which we denote $z$, which will annihilate the poles of $L_{\alpha}$ at all cusps away from some fixed cusp---we will choose $[0]_6$---and one useful Hauptmodul at the cusp, which we will denote $x$.  In the case of $\Gamma_0(6)$ we are lucky, since we can compute a function $z$ which is also a Hauptmodul.

If we examine the lower bounds for the poles of $L_1$ using Theorem \ref{orderboundmodfunc}, then we find that

\begin{align}
\mathrm{ord}_{\infty}^{(6)}(L_1) &\ge 1/3,\\
\mathrm{ord}_{1/3}^{6}(L_1) &\ge 4/3,\\
\mathrm{ord}_{1/2}^{(6)}(L_1) &\ge -1,\\
\mathrm{ord}_{0}^{(6)}(L_1) &\ge -4.
\end{align}  We therefore need our function $z$ to have positive order at the cusp $[1/2]_6$.

A suitable function is the eta quotient

\begin{align}
z = z(\tau) :=& \frac{(q^2;q^2)_{\infty}^9(q^3;q^3)_{\infty}^3}{(q;q)_{\infty}^9(q^6;q^6)_{\infty}^3}\label{zdef}.
\end{align}  The orders of $z$ at the cusps of $\mathrm{X}_0(6)$ can be computed by Theorem \ref{Ligozat}:

\begin{align}
\mathrm{ord}_{\infty}^{(6)}(z) &= 0,\\
\mathrm{ord}_{1/3}^{6}(z) &= 0,\\
\mathrm{ord}_{1/2}^{(6)}(z) &= 1,\\
\mathrm{ord}_{0}^{(6)}(z) &= -1.
\end{align}  This ensures that $zL_1\in\mathcal{M}^{0}\left( \Gamma_0(6) \right)$.

Notice that $z$ also happens to be a Hauptmodul at the cusp $[0]_6$.  However, much more interesting is the fact that, since $(1-q)^9\equiv (1-q^3)^3\pmod{9}$, we have

\begin{align*}
z\equiv 1\pmod{9}.
\end{align*}  Therefore, if we take

\begin{align*}
x = (z-1)/9,
\end{align*} which is also a Hauptmodul with integer coefficients, then the factor of $z=1+9x$ will not affect any common factor that $z\cdot L_1$ has in terms of $x$.  In particular, if $3|z\cdot L_1$, then $3| L_1$.

We make a brief digression to discuss how to compute functions akin to $z$ and $x$ in more general cases.  Were we to work over $\Gamma_0(2\ell)$ in which $\ell$ is an odd prime, we can always construct an eta quotient, with a similar form to that of $z$, which is congruent to $1\pmod{\ell}$.  For $\ell\ge 5$ a good candidate function is

\begin{align*}
\frac{(q^2;q^2)_{\infty}^{\ell}(q^{\ell};q^{\ell})_{\infty}}{(q;q)_{\infty}^{\ell}(q^{2\ell};q^{2\ell})_{\infty}}.
\end{align*}  In the cases that $\ell=3,5$, the functions we have given are Hauptmoduln at $[0]_{2\ell}$.  In this manner, we can easily construct the analogous function to $x$ by subtracting 1 from the analogue of $z$ and dividing by $\ell$.  For $\ell\ge 7$ we will have $\mathfrak{g}\left( \mathrm{X}_0(2\ell) \right)\ge 1$, so that the analogues of $z$ and $x$ are no longer Hauptmoduln, and the question of properly representing the associated functions $L_{\alpha}$ is necessarily more difficult.

Returning to the case of $\Gamma_0(6)$, it is interesting to note that $x$ is also an eta quotient---indeed, the function defined in (\ref{xdef}).  By Theorem \ref{Ligozat}, the orders of $x$ are:

\begin{align}
\mathrm{ord}_{\infty}^{(6)}(x) &= 1,\label{Ligoxn1a}\\
\mathrm{ord}_{1/3}^{6}(x) &= 0,\\
\mathrm{ord}_{1/2}^{(6)}(x) &= 0,\\
\mathrm{ord}_{0}^{(6)}(x) &= -1.\label{Ligoxn1b}
\end{align}  We prove that the two eta quotient representations of $z$ and $x$ are related by

\begin{align}
z = 1+9x\label{zequals1plus9x}
\end{align} in Section \ref{initialrelationssection}.

The functions $x,z$ exhibit certain recurrences which can be expressed with a modular equation.

\begin{theorem}
Define
\begin{align*}
a_0(\tau) &= -(81 x^3 + 18 x^2 + x)\\
a_1(\tau) &=-(1296 x^3 + 279 x^2 + 15 x)\\
a_2(\tau) &= -(5184 x^3 + 1080 x^2 + 57 x).
\end{align*}  Then we have

\begin{align}
x^3+\sum_{j=0}^2 a_j(3\tau) x^j = 0.\label{modX}
\end{align}
\end{theorem}

\begin{proof}
Because $x(3\tau)^{-1}$ is a modular function with a pole at only one cusp of its associated modular curve, we may prove this equation using cusp analysis.  See the end of Section \ref{initialrelationssection}.
\end{proof}

\begin{theorem}
Define
\begin{align*}
b_0(\tau) &=-z^3\\
b_1(\tau) &=1 + 9 z + 9 z^2 - 16 z^3\\
b_2(\tau) &=-2 - 9 z + 72 z^2 - 64 z^3\\
b_3(\tau) &=1.
\end{align*}  Then we have

\begin{align}
\sum_{k=0}^3 b_k(3\tau) z^k = 0.\label{modZ}
\end{align}
\end{theorem}  \noindent Note: our notation of $b_3 =1$ will be useful in later computations.

\begin{proof}
Substitute $x=(z-1)/9$ into (\ref{modX}) and simplify.
\end{proof}

\section{Algebra Structure}\label{algebrastructure}

\subsection{Localized Ring}

We will construct the algebra structure needed for our proof.  Define the multiplicatively closed set

\begin{align}
\mathcal{S} := \left\{ (1+9x)^n : n\in\mathbb{Z}_{n\ge 0} \right\}.\label{Sdef}
\end{align}

We want to prove that for every $\alpha\ge 1$, $L_{\alpha}$ is a member of the localization of $\mathbb{Z}[x]$ at ${\mathcal{S}}$, which we will denote by $\mathbb{Z}[x]_{\mathcal{S}}$.

We need to define two subsets of $\mathbb{Z}[x]_{\mathcal{S}}$.  These will utilize discrete functions:

\begin{definition}
Let $n\ge 1$.  A function $s:\mathbb{Z}\rightarrow\mathbb{Z}$ is discrete if $s(m)=0$ for sufficiently large $m$.  A function $h:\mathbb{Z}^{n}\rightarrow\mathbb{Z}$ is a discrete array if for any fixed $(m_1,m_2,...,m_{n-1})\in\mathbb{Z}^{n-1}$, $h(m_1,m_2,...,m_{n-1},m)$ is discrete with respect to $m$.
\end{definition}

We now construct sets that will contain $L_{\alpha}$.

\begin{align*}
\theta(m) &:= \begin{cases} 
   0, &  1\le m\le 3, \\
   2, & 4\le m\le 6, \\
   \left\lfloor \frac{3m-3}{4} \right\rfloor - 1,       & m\ge 7,
  \end{cases}
\end{align*}

We take some $n\ge 1$, and define the following:

\begin{align}
\mathcal{V}^{(0)}_n:=& \left\{ \frac{1}{(1+9x)^n}\sum_{m\ge 1} s(m)\cdot 3^{\theta(m)}\cdot x^m \right\},\\
\mathcal{V}^{(1)}_n:=& \left\{ \frac{1}{(1+9x)^n}\sum_{m\ge 1} s(m)\cdot 3^{\theta(m)}\cdot x^m : \sum_{m=1}^3s(m)\equiv 0\bmod 9 \right\}.
\end{align}  Notice from (\ref{L1S}) that

\begin{align}
\frac{1}{3}L_1\in\mathcal{V}^{(1)}_1.
\end{align}  Moreover, we must take note of the fact that $\mathcal{V}^{(1)}_n\subset \mathcal{V}^{(0)}_n$ for every $n\ge 1$.

\subsection{Recurrence Relation}

We can use our modular equations to describe how the $U^{(i)}$ operators affect elements of $\mathcal{V}^{(i)}_n$ through certain recurrence relations.

\begin{lemma}
For all $m,n\in\mathbb{Z}$, and $i\in\{0,1\}$, we have
\begin{align}
U^{(i)}\left( \frac{x^m}{(1+9x)^n} \right) = -\frac{1}{(1+9x)^3}\sum_{j=0}^{2}\sum_{k=1}^{3} a_j(\tau)b_k(\tau)\cdot U^{(i)}\left( \frac{x^{m+j-3}}{(1+9x)^{n-k}} \right).\label{UmodX}
\end{align}
\end{lemma}

\begin{proof}
We can write 

\begin{align}
b_0(3\tau) &= -\sum_{k=1}^3 b_k(3\tau) z^k,\nonumber \\
1 &=-\frac{1}{b_0(3\tau)} \sum_{k=1}^3 b_k(3\tau) z^k,\nonumber\\
z^{-n} &=-\frac{1}{b_0(3\tau)} \sum_{k=1}^3 b_k(3\tau) z^{-(n-k)},\label{xnegn1}
\end{align} for $n\ge 1$.  Writing $z$ in terms of $x$, we have

\begin{align}
(1+9x)^{-n} &=-\frac{1}{b_0(3\tau)} \sum_{k=1}^3 b_k(3\tau) (1+9x)^{-(n-k)}.\label{xnegn2}
\end{align}  Multiplying both sides by $x^m$ for some $m\ge 1$, then

\begin{align}
\frac{x^m}{(1+9x)^n} &= -\frac{1}{b_0(3\tau)}\sum_{k=1}^{3} b_k(3\tau)\cdot \frac{x^{m}}{(1+9x)^{n-k}}\nonumber\\
&= -\frac{1}{(1+9y(3\tau))^3}\sum_{k=1}^{3} b_k(3\tau)\cdot \frac{x^{m}}{(1+9x)^{n-k}}.\label{xnegn3}
\end{align}  Now we use the modular equation for $x$:

\begin{align}
\frac{x^m}{(1+9x)^n} &= -\frac{1}{b_0(3\tau)}\sum_{k=1}^{3} b_k(3\tau)\cdot\sum_{j=0}^2 a_j(3\tau) \frac{x^{m+j-3}}{(1+9x)^{n-k}}\nonumber\\
&= -\frac{1}{(1+9x(3\tau))^3}\sum_{j=0}^2 \sum_{k=1}^{3}a_j(3\tau) b_k(3\tau)\cdot \frac{x^{m+j-3}}{(1+9x)^{n-k}}.\label{xnegn4}
\end{align}  Multiply both sides by $\mathcal{A}^{1-i}$:

\begin{align}
\mathcal{A}^{1-i}\cdot\frac{x^m}{(1+9x)^n} &= -\frac{1}{b_0(3\tau)}\sum_{k=1}^{3} b_k(3\tau)\cdot\sum_{j=0}^2 a_j(3\tau) \frac{x^{m+j-3}}{(1+9x)^{n-k}}\nonumber\\
= -\frac{1}{(1+9x(3\tau))^3}&\sum_{j=0}^2\sum_{k=1}^{3}a_j(3\tau) b_k(3\tau)\cdot \mathcal{A}^{1-i}\cdot \frac{x^{m+j-3}}{(1+9x)^{n-k}}.\label{xnegn5}
\end{align}

Recall that by Line 2 of Lemma \ref{impUprop}, for any functions $f(\tau),g(\tau)$,

\begin{align*}
U_3(f(3\tau)\cdot g(\tau)) = f(\tau)\cdot U_3(g(\tau)).
\end{align*} Because of this, we have

\begin{align}
U_3&\left(\mathcal{A}^{1-i}\cdot \frac{x^m}{(1+9x)^n} \right)\nonumber\\ =& -\frac{1}{(1+9x)^3}\sum_{j=0}^{2}\sum_{k=1}^{3} a_j(\tau)b_k(\tau)\cdot U_3\left(\mathcal{A}^{1-i}\cdot \frac{x^{m+j-3}}{(1+9x)^{n-k}} \right).\label{xnegn6}
\end{align}

\end{proof}

\subsection{The Action of $U^{(i)}$ on $\mathcal{V}^{(0)}_n$}

We are going to build certain general relations for $U^{(i)}\left( \frac{x^m}{(1+9x)^n} \right)$.  For this we will define the following arithmetic functions:

\begin{align*}
\pi_0(m,r) &:= \mathrm{max}\left( 0, \left\lfloor \frac{3r-m}{4} \right\rfloor - 1 \right),\\
\pi_1(m,r) &:= \begin{cases} 
   0, &  1\le m\le 3 \text{ and } r=1, \\
   \left\lfloor \frac{3r+1}{4} \right\rfloor,       & 1\le m\le 3 \text{ and } r\ge 2,\\
   \left\lfloor \frac{3r-m+1}{4} \right\rfloor, & m\ge 4
  \end{cases}
\end{align*}  We also define the useful function

\begin{align}
\phi(l) :&=\begin{cases}
& 1,\ l=0\\
&\displaystyle{\left\lfloor \frac{3l+12}{4} \right\rfloor},\ l>0.
\end{cases}
\end{align}  We first prove an important set of inequalities which will become important shortly.

\begin{lemma}\label{piphippibound}
For $l>0$, we have
\begin{align}
\pi_i(m,r-l) + \phi(l) - \pi_i(m,r)\ge 2.\label{inexeq}
\end{align}  for $l=0$,
\begin{align}
\pi_i(m,r-l) + \phi(l) - \pi_i(m,r)= 1.\label{inexeql0}
\end{align}
\end{lemma}

\begin{proof}
We begin by considering $l=0$.  We have $\pi_i(m,r-l)=\pi_i(m,r)$, and 

\begin{align}
\pi_i(m,r-0) +& \phi(0) - \pi_i(m,r)\\
&= \phi(0)= 1.
\end{align}  Henceforth, let us assume that $l>0$.\\

\noindent \underline{\textbf{$\mathbf{i=0}$}}\\

\noindent For $\pi_0(m,r-l)=\left\lfloor \frac{3(r-l)-m}{4} \right\rfloor - 1$ and $\pi_0(m,r)=\left\lfloor \frac{3r-m}{4} \right\rfloor - 1$, we have

\begin{align}
\pi_0(m,r-l) +& \phi(l) - \pi_0(m,r)\\
&\ge \left\lfloor \frac{3(r-l)-m}{4} \right\rfloor - 1 + \left\lfloor \frac{3l+12}{4} \right\rfloor - \left\lfloor \frac{3r-m}{4} \right\rfloor + 1\\
&\ge \left\lfloor \frac{3r-m+9}{4} \right\rfloor - \left\lfloor \frac{3r-m}{4} \right\rfloor\\
&\ge 2.
\end{align}

\noindent For $\pi_0(m,r-l)=0$ and $\pi_0(m,r)=\left\lfloor \frac{3r-m}{4} \right\rfloor - 1 >0$, we must have $\left\lfloor \frac{3(r-l)-m}{4} \right\rfloor\le 1$.  We therefore have

\begin{align}
\pi_0(m,r-l) +& \phi(l) - \pi_0(m,r)\\
&= \left\lfloor \frac{3l+12}{4} \right\rfloor - \left\lfloor \frac{3r-m}{4} \right\rfloor + 1\\
&= \left\lfloor \frac{3l}{4} \right\rfloor + 3 - \left\lfloor \frac{3(r-l)-m}{4}+\frac{3l}{4} \right\rfloor + 1\\
&\ge \left\lfloor \frac{3l}{4} \right\rfloor + 3 - \left\lfloor \frac{3(r-l)-m}{4} \right\rfloor - \left\lfloor\frac{3l}{4} \right\rfloor - 1 + 1\\
&\ge 3 - \left\lfloor \frac{3(r-l)-m}{4} \right\rfloor\\
&\ge 2.
\end{align}

\noindent For $\pi_0(m,r-l)=\pi_0(m,r)=0$, we have

\begin{align}
\pi_0(m,r-l) +& \phi(l) - \pi_0(m,r) = \phi(l)\ge 3.
\end{align}

\noindent \underline{\textbf{$\mathbf{i=1}$}}\\

\noindent For $m\ge 4$ we have

\begin{align}
\pi_1(m,r-l) +& \phi(l) - \pi_1(m,r)\\
&\ge \left\lfloor \frac{3(r-l)-m+1}{4} \right\rfloor + \left\lfloor \frac{3l+12}{4} \right\rfloor - \left\lfloor \frac{3r-m+1}{4} \right\rfloor\\
&\ge \left\lfloor \frac{3r-m+10}{4} \right\rfloor - \left\lfloor \frac{3r-m+1}{4} \right\rfloor\\
&\ge 2.
\end{align}

\noindent For $1\le m\le 3$ and $r-l\ge 2$, we have

\begin{align}
\pi_1(m,r-l) +& \phi(l) - \pi_1(m,r)\\
&\ge \left\lfloor \frac{3(r-l)+1}{4} \right\rfloor + \left\lfloor \frac{3l+12}{4} \right\rfloor - \left\lfloor \frac{3r+1}{4} \right\rfloor\\
&\ge \left\lfloor \frac{3r+10}{4} \right\rfloor - \left\lfloor \frac{3r+1}{4} \right\rfloor\\
&\ge 2.
\end{align}

\noindent For $1\le m\le 3$, $r\ge 2$ and $r-l\le 1$, we have $l\ge r-1$, and therefore

\begin{align}
\pi_1(m,r-l) +& \phi(l) - \pi_1(m,r)\\
&\ge 0 + \left\lfloor \frac{3l+12}{4} \right\rfloor - \left\lfloor \frac{3r+1}{4} \right\rfloor\\
&\ge \left\lfloor \frac{3r+9}{4} \right\rfloor - \left\lfloor \frac{3r+1}{4} \right\rfloor\\
&\ge 2
\end{align}

\noindent For $1\le m\le 3$, $r\le 2$, we have

\begin{align}
\pi_1(m,r-l) +& \phi(l) - \pi_1(m,r)=\phi(l)\ge 2.
\end{align}
\end{proof}  We now state the key theorem on the behavior of $U^{(i)}$ on elements of $\mathcal{V}^{(0)}_n$.

\begin{theorem}\label{thmuyox}
There exist discrete arrays $h_1, h_0 : \mathbb{Z}^{3}\rightarrow\mathbb{Z}$ such that\\
\begin{align*}
U^{(1)}&\left( \frac{x^m}{(1+9x)^n} \right) = \frac{1}{(1+9x)^{3n}}\sum_{r\ge \left\lceil m/3\right\rceil} h_1(m,n,r)\cdot 3^{\pi_1(m,r)}\cdot x^r,\\
U^{(0)}&\left( \frac{x^m}{(1+9x)^n} \right) = \frac{1}{(1+9x)^{3n+1}}\sum_{r\ge \left\lceil (m+1)/3\right\rceil} h_0(m,n,r)\cdot 3^{\pi_0(m,r)}\cdot x^r.
\end{align*}
\end{theorem}

We prove this theorem by first demonstrating the following two lemmas.  Because the proofs work identically whether we apply $U^{(1)}$ or $U^{(0)}$, we denote the differing denominator powers as $3n+\kappa$, lower bounds for $r$ by $\left\lceil (m+\delta)/3\right\rceil$, and powers of 3 by $\left\lfloor \frac{3r-m+\mu}{4} \right\rfloor$.  We define $\kappa=1,0$, $\delta=1,0$, $\mu=-4,1$ depending on whether $i=0,1$.

\begin{lemma}\label{lemmaMN}
Let $\kappa,\delta,\mu\in\mathbb{Z}_{\ge 0}$ be fixed, and fix $i$ to either 0 or 1.  If there exists a discrete array $h_i$ such that
\begin{align}
U^{(i)}\left( \frac{x^{m}}{(1+9x)^{n}} \right) &= \frac{1}{(1+9x)^{3n+\kappa}}\sum_{r\ge \left\lceil (m+\delta)/3 \right\rceil} h_i(m,n,r)\cdot 3^{\left\lfloor \frac{3r-m+\mu}{4} \right\rfloor}\cdot x^{r}
\end{align} for $1\le m\le 3$, $1\le n\le 3$, then such a relation can be made to hold for all $m\ge 1$, $n\ge 1$.
\end{lemma}

\begin{lemma}\label{lemmaN}
Let $\kappa,\delta\in\mathbb{Z}_{\ge 0}$ and $m\in\mathbb{Z}_{\ge 1}$ be fixed, and fix $i$ to either 0 or 1.  If there exists a discrete array $h_i$ such that
\begin{align}
U^{(i)}\left( \frac{x^{m}}{(1+9x)^{n}} \right) &= \frac{1}{(1+9x)^{3n+\kappa}}\sum_{r\ge \left\lceil (m+\delta)/3 \right\rceil} h_i(m,n,r)\cdot 3^{\pi_i(m,r)}\cdot x^{r}
\end{align} for $1\le n\le 3$, then such a relation can be made to hold for all $n\ge 1$.
\end{lemma}  The latter lemma has the advantage of allowing us to more closely study the arithmetic properties of the coefficients $h_i(m,n,r)$.

\begin{proof}[Proof of Lemma \ref{lemmaMN}]

We use induction.  Suppose that our relation holds for all positive integers strictly less than some $m_0,n_0\in\mathbb{Z}_{\ge 6}$.  We show that the relation will hold for $m_0$ and $n_0$.  We have
\begin{align}
U^{(i)}\left( \frac{x^{m_0}}{(1+9x)^{n_0}} \right)&\nonumber\\ = -\frac{1}{(1+9x)^3}&\sum_{j=0}^{2}\sum_{k=1}^{3} a_j(\tau)b_k(\tau)\cdot U^{(i)}\left( \frac{x^{m_0+j-3}}{(1+9x)^{n_0-k}} \right)\\
= -\frac{1}{(1+9x)^3}\sum_{j=0}^{2}&\sum_{k=1}^{3} \frac{a_j(\tau)b_k(\tau)}{(1+9x)^{3(n_0-k)+\kappa}}\nonumber\\ \times\sum_{r\ge \left\lceil (m_0+j-3 + \delta)/3 \right\rceil}& h_i(m_0+j-3,n_0-k,r)\cdot 3^{\left\lfloor \frac{3r-(m_0+j-3)+\mu}{4} \right\rfloor}\cdot x^r\\
=\frac{1}{(1+9x)^{3n_0+\kappa}}&\sum_{j=0}^{2}\sum_{k=1}^{3} w(j,k)\nonumber\\ \times\sum_{r\ge \left\lceil (m_0+j-3+\delta)/3 \right\rceil} &h_i(m_0+j-3,n_0-k,r)\cdot 3^{\left\lfloor \frac{3r-(m_0+j-3)+\mu}{4} \right\rfloor}\cdot x^r,
\end{align} with

\begin{align}
w(j,k) :=& -a_j(\tau)b_k(\tau)(1+9x)^{3(k-1)}\nonumber\\
=& \sum_{l=1}^{12} v(j,k,l)\cdot 3^{\left\lfloor \frac{3l+j}{4} \right\rfloor}\cdot x^l.
\end{align}  This can be demonstrated by an expansion of $a_j(\tau)b_k(\tau)(1+9x)^{3(k-1)}$ which is detailed in the Mathematica supplement.  We have

\begin{align}
&U^{(i)}\left( \frac{x^{m_0}}{(1+9x)^{n_0}} \right)\nonumber\\ =& \frac{1}{(1+9x)^{3n_0+\kappa}}\sum_{j=0}^{2}\sum_{k=1}^{3}\sum_{l=1}^{12} \sum_{r\ge \left\lceil (m_0+j-3+\delta)/3 \right\rceil}\nonumber\\ &v(j,k)\cdot h_i(m_0+j-3,n_0-k,r)\cdot 3^{\left\lfloor \frac{3r-(m_0+j-3)+\mu}{4} \right\rfloor + \left\lfloor \frac{3l+j}{4} \right\rfloor}\cdot x^{r+l}.\label{mnmodeqnrpl}
\end{align}  Notice that for any $M,N\in\mathbb{Z}$,

\begin{align*}
\left\lfloor \frac{M}{4} \right\rfloor + \left\lfloor \frac{N}{4} \right\rfloor \ge \left\lfloor \frac{M+N-3}{4} \right\rfloor.
\end{align*}  Because of this,

\begin{align}
\left\lfloor \frac{3r-(m_0+j-3)+\mu}{4} \right\rfloor + \left\lfloor \frac{3l+j}{4} \right\rfloor \ge \left\lfloor \frac{3(r+l) - m_0 + \mu}{4} \right\rfloor.
\end{align} And because

\begin{align}
r+l &\ge\left\lceil \frac{m_0+j-3+\delta}{3} \right\rceil + l\\ &\ge \left\lceil \frac{m_0+\delta}{3}-\frac{3-j}{3} \right\rceil + l\\
&\ge \left\lceil \frac{m_0+\delta}{3} \right\rceil - 1 + l\\ &\ge \left\lceil \frac{m_0+\delta}{3} \right\rceil,
\end{align} we can relabel our powers of $x$ so that

\begin{align}
&U^{(i)}\left( \frac{x^{m_0}}{(1+9x)^{n_0}} \right)\nonumber\\ =& \frac{1}{(1+9x)^{3n_0+\kappa}}\sum_{\substack{0\le j\le 2,\\ 1\le k\le 3,\\ 1\le l\le 12}} \sum_{r\ge \left\lceil \frac{m_0+\delta}{3} \right\rceil}\nonumber\\ &v(j,k)\cdot h_i(m_0+j-3,n_0-k,r-l)\cdot 3^{\left\lfloor \frac{3r-(m_0+j-3)+\mu}{4} \right\rfloor + \left\lfloor \frac{3l+j}{4} \right\rfloor}\cdot x^{r}.
\end{align}  If we define the discrete array $H_i$ by

\begin{align}
H_i(m,n,r) := \begin{cases}
\sum_{\substack{0\le j\le 2,\\ 1\le k\le 3,\\ 1\le l\le 12}} &\sum_{r\ge \left\lceil \frac{m+\delta}{3} \right\rceil - 1 + l} \hat{H}(i,j,k,l,r),\ r\ge l\\
0,\ r<l
\end{cases}\label{Hi}
\end{align} with

\begin{align*}
\hat{H}(i,j,k,l,r) := v(j,k)\cdot h_i(m+j-3,n-k,r-l)\cdot 3^{\epsilon(i,j,l,m,r)},
\end{align*}

\begin{align*}
\epsilon(i,j,l,m,r) := \left\lfloor \frac{3(r-l)-(m+j-3)+\mu}{4} \right\rfloor + \left\lfloor \frac{3l+j}{4} \right\rfloor -\left\lfloor \frac{3r - m_0 + \mu}{4} \right\rfloor,
\end{align*} then

\begin{align}
U^{(i)}\left( \frac{x^{m_0}}{(1+9x)^{n_0}} \right) &= \frac{1}{(1+9x)^{3n_0+\kappa}}\nonumber\\ &\times\sum_{r\ge \left\lceil \frac{m_0+\delta}{3} \right\rceil} H_i(m_0,n_0,r)\cdot 3^{\left\lfloor \frac{3r-m_0+\mu}{4} \right\rfloor}\cdot x^{r}.
\end{align}

\end{proof}

\begin{proof}[Proof of Lemma \ref{lemmaN}]

We begin in a manner analogous to that of the previous lemma.  However, we only want a recursion for the denominator of each term; we fix $m$.

\begin{align}
U^{(i)}&\left( \frac{x^{m}}{(1+9x)^{n}} \right)\nonumber\\ =& -\frac{1}{(1+9x)^3}\sum_{k=1}^{3}b_k(\tau)\cdot U^{(i)}\left( \frac{x^{m}}{(1+9x)^{n-k}} \right)\\
=& -\frac{1}{(1+9x)^3}\sum_{k=1}^{3} \frac{b_k(\tau)}{(1+9x)^{3(n-k)-\kappa}} \sum_{r\ge 1} h_i(m,n-k,r)\cdot 3^{\pi_i(m,r)}\cdot x^r\\
=&\frac{1}{(1+9x)^{3n+\kappa}}\sum_{k=1}^{3} \hat{w}(k)\sum_{r\ge 1 } h_i(m,n-k,r)\cdot 3^{\pi_i(m,r)}\cdot x^r,
\end{align} with

\begin{align}
\hat{w}(k) :&= -b_k(\tau)(1+9x)^{3(k-1)}\nonumber\\
&=\begin{cases}
& \displaystyle{\sum_{l=0}^{6} \hat{v}(k,l)\cdot 3^{\phi(l)}\cdot x^l},\ k<3\\
&\displaystyle{1+\sum_{l=1}^{6} \hat{v}(3,l)\cdot 3^{\phi(l)}\cdot x^l},\ k=3.
\end{cases}
\end{align}  Similar to $w(i,j)$ above, this is a straightforward expansion of $\hat{w}(k)$ that we detail in our Mathematica supplement.  Expanding, we have

\begin{align}
U^{(i)}&\left( \frac{x^{m}}{(1+9x)^{n}} \right) = \frac{1}{(1+9x)^{3n+\kappa}}\nonumber\\
\times&\Bigg(\sum_{\substack{1\le k\le 2,\\ 0\le l\le 6,\\ r\ge \left\lceil \frac{m+\delta}{3} \right\rceil}} \hat{v}(k,l)\cdot h_i(m,n-k,r)\cdot 3^{\pi_i(m,r) + \phi(l)}\cdot x^{r+l}\label{reliancenm5a3}\\
&+\sum_{\substack{1\le l\le 6,\\ r\ge \left\lceil \frac{m+\delta}{3} \right\rceil}} \hat{v}(3,l)\cdot h_i(m,n-3,r)\cdot 3^{\pi_i(m,r) + \phi(l)}\cdot x^{r+l}\label{reliancenm5a2}\\
&+\sum_{r\ge \left\lceil \frac{m+\delta}{3} \right\rceil} h_i(m,n-3,r)\cdot 3^{\pi_i(m,r)}\cdot x^{r}\Bigg).\label{reliancenm5a}
\end{align}  With a change of index, we have

\begin{align}
U^{(i)}\left( \frac{x^{m}}{(1+9x)^{n}} \right) = \frac{1}{(1+9x)^{3n+\kappa}}
\times&\Bigg(\sum_{\substack{1\le k\le 2,\\ 0\le l\le 6,\\ r\ge l+ \left\lceil \frac{m+\delta}{3} \right\rceil}} \hat{v}(k,l)\cdot h_i(m,n-k,r-l)\cdot 3^{\pi_i(m,r-l) + \phi(l)}\cdot x^{r}\label{reliancenm5aaaa}\\
&+\sum_{\substack{1\le l\le 6,\\ r\ge l+ \left\lceil \frac{m+\delta}{3} \right\rceil}}\hat{v}(3,l)\cdot h_i(m,n-3,r-l)\cdot 3^{\pi_i(m,r-l) + \phi(l)}\cdot x^{r}\label{reliancenm5aaaaa}\\
&+\sum_{r\ge \left\lceil \frac{m+\delta}{3} \right\rceil} h_i(m,n-3,r)\cdot 3^{\pi_i(m,r)}\cdot x^{r}\Bigg).\label{reliancenm5aaaaaa}
\end{align}  To finish the proof, we note from Lemma \ref{piphippibound} that $\pi_i(m,r-l) + \phi(l) - \pi_i(m,r)\ge 1$.  This ensures that the critical 3-adic valuation of the terms of $U^{(i)}\left( \frac{x^{m}}{(1+9x)^{n}} \right)$ derives precisely from the sum

\begin{align}
\frac{1}{(1+9x)^{3n+\kappa}}&\sum_{r\ge \left\lceil \frac{m+\delta}{3} \right\rceil} h_i(m,n-3,r)\cdot 3^{\pi_i(m,r)}\cdot x^{r}\nonumber\\ &= U^{(i)}\left( \frac{x^{m}}{(1+9x)^{n-3}} \right).\label{reliancenm5b}
\end{align}  We rearrange the sum and define a discrete array in the same manner that we defined (\ref{Hi}).

\end{proof}

\begin{proof}[Proof of Theorem \ref{thmuyox}]

Lemma \ref{lemmaMN} is sufficient to prove Theorem \ref{thmuyox} for $m\ge 4$.  However, for $1\le m\le 3$, we have larger values of $\pi_i(m,r)$ than the simple floor function.

To account for these larger values, we also use Lemma \ref{lemmaN} to prove Theorem \ref{thmuyox} for $1\le m\le 3$.

We need to establish that the theorem applies in the explicit cases $1\le m\le 3,$ $1\le n\le 3$.  We show how to verify these cases in Section \ref{initialrelationssection}.

\end{proof}

As a consequence of Lemmas \ref{piphippibound}, \ref{lemmaN}, we can derive results on the arithmetic properties of the functions $h_i$:

\begin{corollary}\label{congcor1}
For all $n\in\mathbb{Z}_{\ge 1}$ we have:
\begin{align}
h_i(m,n,r)&\equiv h_i(m,n-3,r)\pmod{3}.\label{hmod2}
\end{align}  In particular, for all $n\in\mathbb{Z}_{\ge 1}$ and $1\le m\le 3$ we have:
\begin{align}
h_1(m,1,2)&\equiv 0\pmod{3},\label{hmod2b}\\
h_0(m,3,4)&\equiv 0\pmod{3},\label{hmod4}.
\end{align} Finally,
\begin{align}
h_1(m,1,4)&\equiv 0\pmod{3},\label{hmod3}
\end{align} for $1\le m\le 6$.
\end{corollary}

\begin{proof}

We reexamine (\ref{reliancenm5aaaa}), (\ref{reliancenm5aaaaa}), (\ref{reliancenm5aaaaaa}).  Notice that due to Lemma \ref{piphippibound}, if we divide out $3^{\pi_i(m,r)}$, then $h_i(m,n-3,r)$ is the only term which does not vanish modulo 3.

The congruences (\ref{hmod2b}), (\ref{hmod4}), (\ref{hmod3}) may be checked by computation.

\end{proof}

\begin{corollary}\label{congcor2}
For all $m,r\in\mathbb{Z}$ such that $1\le m\le 3r$, $1\le r\le 5$, and $n\in\mathbb{Z}_{\ge 1}$ we have:
\begin{align}
h_1(m,3n+1,r)&\equiv h_1(m,1,r)\pmod{9}.\label{hmod2a}
\end{align}  Moreover, for all $r,w\in\mathbb{Z}$ such that $1\le r\le 3w-1$, $1\le w\le 3$,
\begin{align}
h_0(m,3n,r)&\equiv h_0(m,3,r)\pmod{9}.\label{hmod2aa}
\end{align}  In particular,
\begin{align}
h_1(m,3n+1,1)\equiv 0\pmod{9}.\label{hmod22222}
\end{align}
\end{corollary}

\begin{proof}

Notice that $\pi_i(m,r-l) + \phi(l) - \pi_i(m,r)\ge 2$ except when $l=0$.  On the other hand, $\hat{v}(1,0)= 1$ and $\hat{v}(2,0)=-1$.  We can therefore reduce (\ref{reliancenm5aaaa}), (\ref{reliancenm5aaaaa}), (\ref{reliancenm5aaaaaa}) modulo 9 and examine the coefficient of $x^r$:

\begin{align}
h_i(m,n,r)\equiv& 3\left(h_i(m,n-1,r)-h_i(m,n-2,r)\right) + h_i(m,n-3,r)\pmod{9}.\label{hiequmod3}
\end{align}  For $1\le m\le 3r$, $1\le r\le 5$, we can compute that

\begin{align*}
h_1(m,1,r)\equiv h_1(m,2,r)\equiv h_1(m,3,r)\pmod{3}.
\end{align*}  Similarly, for $1\le r\le 3w-1$, $1\le w\le 3$, we compute that

\begin{align*}
h_0(m,1,r)\equiv h_0(m,2,r)\equiv h_0(m,3,r)\pmod{3}.
\end{align*}  This ensures that $h_i(m,n-1,r)-h_i(m,n-2,r)\equiv 0\pmod{3}$.  Combining this with (\ref{hiequmod3}), we have

\begin{align}
h_i(m,n,r)\equiv& h_i(m,n-3,r)\pmod{9}.\label{hiequmod9}
\end{align}  This of course permits us to reduce $h_1(m,3n+1,r)$ to $h_1(m,1,r)$, and $h_0(m,3n,r)$ to $h_0(m,3,r)$, modulo~9.

The congruence (\ref{hmod22222}) can then be verified by checking the cases $h_1(m,1,1)$, $1\le m\le 3$ computationally.

\end{proof}

\section{Main Theorem}\label{maintheoremsection}

We can now work towards the proof of Theorem \ref{Mythm}.  We start with the simpler question of how to elements of $\mathcal{V}_n^{(0)}$ respond under $U^{(0)}$:

\begin{theorem}\label{v02v1}
Let $f\in\mathcal{V}_n^{(0)}$.  Then
\begin{align}
U^{(0)}\left( f \right)\in\mathcal{V}_{3n+1}^{(0)}.
\end{align}
\end{theorem}

\begin{proof}

Let $f\in\mathcal{V}_n^{(0)}$.  Then we can express $f$ as
\begin{align*}
f = \frac{1}{(1+9x)^n}\sum_{m\ge 1} s(m)\cdot 3^{\theta(m)}\cdot x^m.
\end{align*}  We write

\begin{align}
U^{(0)}\left( f \right) &= \sum_{m\ge 1} s(m)\cdot 3^{\theta(m)}\cdot U^{(0)}\left( \frac{x^m}{(1+9x)^n}\right)\\
&=\frac{1}{(1+9x)^{3n+1}}\sum_{m\ge 1}\sum_{r\ge \left\lceil (m+1)/3 \right\rceil} s(m)\cdot h_0(m,n,r) 3^{\theta(m) + \pi_0(m,r)}\cdot x^r\\
&=\frac{1}{(1+9x)^{3n+1}}\sum_{r\ge 1}\sum_{m\ge 1} s(m)\cdot h_0(m,n,r) 3^{\theta(m) + \pi_0(m,r)}\cdot x^r.
\end{align}  We want to show that

\begin{align}
\theta(m) + \pi_0(m,r) - \theta(r)\ge 0.\label{v02v1nonneg}
\end{align}  For $1\le m\le 6$ and $1\le r\le 6$, it is a simple calculation.  Some exceptions to (\ref{v02v1nonneg}) occur, but they can be handled with relative ease.  See Appendix I.

We consider the remaining cases below:

\subsection{$m=1,2,3$, $r\ge 7$}

\begin{align*}
\theta(m) + \pi_0(m,r) - \theta(r) &= \left\lfloor \frac{3r-m}{4} \right\rfloor - 1 - \left( \left\lfloor \frac{3r-3}{4} \right\rfloor - 1 \right)\\
&= \left\lfloor \frac{3r-m}{4} \right\rfloor - \left\lfloor \frac{3r-3}{4} \right\rfloor\\ &\ge 0.
\end{align*}

\subsection{$m=4,5,6$, $r\ge 7$}

\begin{align*}
\theta(m) + \pi_0(m,r) - \theta(r) &= \left\lfloor \frac{3r-m}{4} \right\rfloor - 1 + 2 - \left( \left\lfloor \frac{3r-3}{4} \right\rfloor - 1 \right)\\
&\ge \left\lfloor \frac{3r-6}{4} \right\rfloor - \left\lfloor \frac{3r-3}{4} \right\rfloor + 2\\
&= \left\lfloor \frac{3r-2}{4} \right\rfloor - \left\lfloor \frac{3r-3}{4} \right\rfloor + 1\\
&\ge 0.
\end{align*}

\subsection{$m\ge 7$, $r\ge 7$}

For $r\ge 7$, we have:

\begin{align*}
\theta(m) + \pi_0(m,r) - \theta(r) &= \left\lfloor \frac{3r-m-4}{4} \right\rfloor + \left\lfloor \frac{3m-3}{4} \right\rfloor - 1 - \left( \left\lfloor \frac{3r-3}{4} \right\rfloor - 1 \right)\\
&\ge \left\lfloor \frac{3r+2m-10}{4} \right\rfloor - \left\lfloor \frac{3r-3}{4} \right\rfloor\\
&= \left\lfloor \frac{3r-3}{4} \right\rfloor + \left\lfloor \frac{2m-7}{4} \right\rfloor - \left\lfloor \frac{3r-3}{4} \right\rfloor\\
&\ge \left\lfloor \frac{7}{4} \right\rfloor \ge 0.
\end{align*}

\subsection{$m\ge 7$, $r\le 6$}

We note that $r\ge \left\lfloor (m+1)/3 \right\rfloor \ge 3$.  We have

\begin{align*}
\theta(m) + \pi_0(m,r) - \theta(r) &= \left\lfloor \frac{3r-m-4}{4} \right\rfloor + \left\lfloor \frac{3m-3}{4} \right\rfloor - 1 - \theta(r)\\
&\ge \left\lfloor \frac{3r+2m-10}{4} \right\rfloor - 1 - \theta(r)\\
&\ge \left\lfloor \frac{-1+2m}{4} \right\rfloor - 3\\
&\ge \left\lfloor \frac{-1+14}{4} \right\rfloor - 3\\
&\ge 0.
\end{align*}

\end{proof}

We now study the analogous question of how elements of $\mathcal{V}_n^{(1)}$ respond under $U^{(1)}$:

\begin{theorem}\label{gofrom1to0}
Let $f\in\mathcal{V}_n^{(1)}$ with $n\equiv 1\pmod{3}$.  Then

\begin{align}
\frac{1}{9}U^{(1)}\left( f \right)&\in\mathcal{V}_{3n}^{(0)}.
\end{align}
\end{theorem}

\begin{proof}
Let $f\in\mathcal{V}_n^{(1)}$ such that
\begin{align*}
f &= \frac{1}{(1+9x)^n}\sum_{m\ge 1} s(m)\cdot 3^{\theta(m)}\cdot x^m.
\end{align*}  We then have

\begin{align*}
U^{(1)}\left( f \right) &= \frac{1}{(1+9x)^{3n}}\sum_{m\ge 1}\sum_{r\ge \left\lceil m/3 \right\rceil} s(m)\cdot h_1(m,n,r)\cdot 3^{\theta(m) + \pi_1(m,r)}\cdot x^r\\
&= \frac{1}{(1+9x)^{3n}}\sum_{r\ge 1}\sum_{m\ge 1} s(m)\cdot h_1(m,n,r)\cdot 3^{\theta(m) + \pi_1(m,r)}\cdot x^r
\end{align*}  We want to show that

\begin{align}
\theta(m) + \pi_1(m,r) - \theta(r) - 2 \ge 0.\label{v12v0nonneg}
\end{align}  For $1\le m\le 6$ and $2\le r\le 6$, it is a simple calculation.  In fact, exceptions to (\ref{v12v0nonneg}) do occur, but we can successfully compensate for them.  See Appendix I.

We handle the remaining cases as follows:

\subsection{$m=1,2,3$, $r\ge 7$}

\begin{align*}
\theta(m) + \pi_1(m,r) - \theta(r) - 2 &= \left\lfloor \frac{3r+1}{4} \right\rfloor - \left( \left\lfloor \frac{3r-3}{4} \right\rfloor - 1 \right) - 2\\
&= \left\lfloor \frac{3r+1}{4} \right\rfloor - \left\lfloor \frac{3r-3}{4} \right\rfloor - 1 \\
&= \left\lfloor \frac{3r-3}{4} \right\rfloor + 1 - \left\lfloor \frac{3r-3}{4} \right\rfloor - 1\\
&\ge 0.
\end{align*}

\subsection{$m=4,5,6$, $r\ge 7$}

\begin{align*}
\theta(m) + \pi_1(m,r) - \theta(r) - 2 &\ge \left\lfloor \frac{3r-1}{4} \right\rfloor - 1 + 2 - \left( \left\lfloor \frac{3r-3}{4} \right\rfloor - 1 \right) - 2\\
&= \left\lfloor \frac{3r-1}{4} \right\rfloor - \left\lfloor \frac{3r-3}{4} \right\rfloor \\ &\ge 0.
\end{align*}

\subsection{$m\ge 7$, $3\le r\le 6$}

For $r=3$,

\begin{align*}
\theta(m) + \pi_1(m,r) - \theta(r) - 2 &= \left\lfloor \frac{3r-m+1}{4} \right\rfloor + \left\lfloor \frac{3m-3}{4} \right\rfloor - 1 - 2\\
&\ge \left\lfloor \frac{3r+2m-5}{4} \right\rfloor - \left\lfloor \frac{3r-3}{4} \right\rfloor - 3\\
&\ge \left\lfloor \frac{4+14}{4} \right\rfloor - \left\lfloor \frac{6}{4} \right\rfloor - 3\\
&= 0.
\end{align*}

For $4\le r\le 6$, we have:

\begin{align*}
\theta(m) + \pi_1(m,r) - \theta(r) - 2 &= \left\lfloor \frac{3r-m+1}{4} \right\rfloor + \left\lfloor \frac{3m-3}{4} \right\rfloor - 1 - 2 - 2\\
&\ge \left\lfloor \frac{3r+2m-5}{4} \right\rfloor - 5\\
&\ge \left\lfloor \frac{7+14}{4} \right\rfloor - 5\\
&\ge 0.
\end{align*}

\subsection{$m\ge 7$, $r\ge 7$}

For $r\ge 7$, we have:

\begin{align*}
\theta(m) + \pi_1(m,r) - \theta(r) - 2 &= \left\lfloor \frac{3r-m+1}{4} \right\rfloor + \left\lfloor \frac{3m-3}{4} \right\rfloor - 1 - \left( \left\lfloor \frac{3r-3}{4} \right\rfloor - 1 \right) - 2\\
&\ge \left\lfloor \frac{3r+2m-5}{4} \right\rfloor - \left\lfloor \frac{3r-3}{4} \right\rfloor - 2\\
&\ge \left\lfloor \frac{3r-3}{4} \right\rfloor + \left\lfloor \frac{2m-2}{4} \right\rfloor - \left\lfloor \frac{3r-3}{4} \right\rfloor - 2\\
&= \left\lfloor \frac{2m-2}{4} \right\rfloor - 2\\
&\ge 3-2 \ge 0.
\end{align*}

This finally leaves:

\subsection{$1\le m\le 3$, $r=1$}

\begin{align*}
\theta(m) + \pi_1(m,r) - \theta(r) - 2 &= \left\lfloor \frac{3-m+1}{4} \right\rfloor + 0 - 0 - 2 = -2 <0.
\end{align*}  How do we deal with this dilemma?

If we examine the coefficient for $r=1$, i.e., the coefficient for $x^1$, then we get

\begin{align}
\sum_{m=1}^3 s(m)\cdot h_1(m,n,1)\cdot 3^{\theta(m) + \pi_1(m,1)}.
\end{align}  We have just shown that $\theta(m) + \pi_1(m,1)=0$.  However, remembering from (\ref{hmod22222}) that $h_1(m,n,1)\equiv 1\pmod{9}$, we have

\begin{align}
\sum_{m=1}^3 s(m)\cdot h_1(m,n,1)\cdot 3^{\theta(m) + \pi_1(m,1)}\equiv
\sum_{m=1}^3 s(m)\equiv 0\pmod{9},
\end{align} since $f\in\mathcal{V}_n^{(1)}$.

\end{proof}

This shows that our additional condition on $\mathcal{V}_n^{(1)}$ successfully compensates for the apparent failure of the very first term to gain the necessary powers of 3.  However, it is clearly not enough to show that a specific function in $\mathcal{V}_n^{(1)}$ will gain divisibilty by 9 on application of $U^{(1)}$.  It is critical to determine that this condition is \textit{stable}---that is, on applying $U^{(0)}\circ U^{(1)}$ and then dividing out the expected power of 9, what remains ought to be an element of $\mathcal{V}_{9n+1}^{(1)}$.  In particular, the first three coefficients must sum to divisibility by 9.

If we can prove this, then we have enough for an induction proof.  It is this critical result which we supply in the following theorem:

\begin{theorem}\label{oddbacktoodd}
Let $f\in\mathcal{V}_n^{(1)}$, with $n\equiv 1\pmod{3}$.  Then

\begin{align}
\frac{1}{9}U^{(0)}\circ U^{(1)}\left( f \right)&\in\mathcal{V}_{9n+1}^{(1)}.
\end{align}
\end{theorem}

\begin{proof}

We have already established that

\begin{align*}
\frac{1}{9}U^{(1)}\left( f \right)&\in\mathcal{V}_{3n}^{(0)},
\end{align*} and that

\begin{align*}
\frac{1}{9}U^{(0)}\circ U^{(1)}\left( f \right)&\in\mathcal{V}_{9n+1}^{(0)}.
\end{align*}  So we know that 

\begin{align*}
\frac{1}{9}U^{(0)}\circ U^{(1)}\left( f \right) &= \frac{1}{(1+9x)^{9n+1}}\sum_{w\ge 1}t(w)\cdot 3^{\theta(w)}\cdot x^w,
\end{align*} with the coefficients $t(w)$ defined by

\begin{align*}
t(w) =& \sum_{r=1}^{3w-1}\sum_{m=1}^{3r}s(m)\cdot h_1(m,n,r)\cdot h_0(r,3n,w)\\ &\times 3^{\theta(m) + \pi_1(m,r) + \pi_0(r,w) -\theta(w) -2}.
\end{align*}  We want to prove that 

\begin{align*}
t(1)+t(2)+t(3)\equiv 0\pmod{9}.
\end{align*}  This will establish that 

\begin{align*}
\frac{1}{9}U^{(0)}\circ U^{(1)}\left( f \right)&\in\mathcal{V}_{9n+1}^{(1)}.
\end{align*}  Notice that for $1\le w\le 3$, $\theta(w)=0$.  So for these cases, we have

\begin{align*}
t(w) =& \sum_{r=1}^{3w-1}\sum_{m=1}^{3r}s(m)\cdot h_1(m,n,r)\cdot h_0(r,3n,w)\\ &\times 3^{\theta(m) + \pi_1(m,r) + \pi_0(r,w) -2}.
\end{align*}

Moreover, we need not consider terms for which $\theta(m) + \pi_1(m,r) + \pi_0(r,w) -2\ge 2$, i.e., terms which will vanish modulo 9.  A quick computation shows that $\theta(m) + \pi_1(m,r) + \pi_0(r,w) \ge 4$ for $1\le m\le 3r$, $1\le w\le 3$, and $r\ge 5$.  We therefore have

\begin{align*}
t(1) \equiv & \sum_{r=1}^{2}\sum_{m=1}^{3r}s(m)\cdot h_1(m,n,r)\cdot h_0(r,3n,w)\\ &\times 3^{\theta(m) + \pi_1(m,r) + \pi_0(r,w) -2}\pmod{9},
\end{align*}  and

\begin{align*}
t(w) \equiv & \sum_{r=1}^{5}\sum_{m=1}^{3r}s(m)\cdot h_1(m,n,r)\cdot h_0(r,3n,w)\\ &\times 3^{\theta(m) + \pi_1(m,r) + \pi_0(r,w) -2}\pmod{9}
\end{align*} for $w=2,3$.  Next, we consider that, by Corollary \ref{congcor2}, we have

\begin{align*}
h_0(m,3n,r)&\equiv h_0(m,3,r)\pmod{9},\\
h_1(m,3n+1,r)&\equiv h_1(m,1,r)\pmod{9},
\end{align*} for $1\le r\le 5$, $1\le m\le 3r$.  We therefore define $\hat{t}(w)$ by

\begin{align*}
\hat{t}(1) :=& \sum_{r=1}^{2}\sum_{m=1}^{3r}s(m)\cdot h_1(m,1,r)\cdot h_0(r,3,w)\\ &\times 3^{\theta(m) + \pi_1(m,r) + \pi_0(r,w) -2},
\end{align*} and

\begin{align*}
\hat{t}(w) := & \sum_{r=1}^{5}\sum_{m=1}^{3r}s(m)\cdot h_1(m,1,r)\cdot h_0(r,3,w)\\ &\times 3^{\theta(m) + \pi_1(m,r) + \pi_0(r,w) -2},
\end{align*} for $2\le w\le 3$.

We have $\hat{t}(w)\equiv t(w)\pmod{9}$ for $1\le w\le 3$.  To prove that $t(1)+t(2)+t(3)\equiv 0\pmod{9}$, we have only to show that $\hat{t}(1)+\hat{t}(2)+\hat{t}(3)\equiv 0\pmod{9}$.

We provide the detailed computations of these terms in our Mathematica supplement.  These computations yield the following:

\begin{align*}
\hat{t}(1) =& -\frac{2s(1)}{3} + \frac{22s(2)}{3} + \frac{49s(3)}{3} + 82s(4) + 16s(5) + s(6),\\
\hat{t}(2) =& \frac{625s(1)}{3} + \frac{36784s(2)}{3} + \frac{1201024s(3)}{3} + 24003457s(4) + 60815254s(5) + 77522425s(6) + 172725210s(7)\\ &+ 243795825s(8) + 228765303s(9) + 48752847s(10) + 21348036s(11) + 6305121s(12) + 1201392s(13)\\ &+ 44469s(14) + 2187s(15),\\
\hat{t}(3) =& \frac{51181s(1)}{3} + \frac{8480416s(2)}{3} + \frac{315495472s(3)}{3} + 24003457s(4) + 60815254s(5) + 77522425s(6)\\ &+ 172725210s(7) + 243795825s(8) + 228765303s(9) + 48752847s(10) + 21348036s(11) + 6305121s(12)\\ &+ 1201392s(13) + 44469s(14) + 2187s(15).
\end{align*}

Notice that the apparently fractional parts of each term above are integers:

\begin{align*}
-\frac{2s(1)}{3} + \frac{22s(2)}{3} + \frac{49s(3)}{3}&\in\mathbb{Z},\\
\frac{625s(1)}{3} + \frac{36784s(2)}{3} + \frac{1201024s(3)}{3}&\in\mathbb{Z},\\
\frac{51181s(1)}{3} + \frac{8480416s(2)}{3} + \frac{315495472s(3)}{3}&\in\mathbb{Z}.
\end{align*}  This is guaranteed by the fact that $s(1)+s(2)+s(3)\equiv 0\pmod{9}$.  Thus, we have $\hat{t}(1),\hat{t}(2),\hat{t}(3)\in\mathbb{Z}$.

Summing the three together and reducing modulo 9, we have

\begin{align*}
\hat{t}(1)+\hat{t}(2)+\hat{t}(3) =& 17268 s(1) + 2839074 s(2) + 105565515 s(3) + 6356507958 s(4) + 16011491148 s(5)\\ &+ 20309056443 s(6) + 45101492226 s(7) + 63531577746s(8) + 59545450458s(9)\\& + 12681870624s(10) + 5551428312s(11) + 1639392696s(12) + 312361920s(13)\\ &+ 11561940s(14) + 568620s(15)\\
\equiv& 6 (s(1)+s(2)+s(3))\\ \equiv& 0\pmod{9}.
\end{align*}

\end{proof}

At last, we have enough to prove Theorem \ref{Mythm}:

\begin{proof}[Proof of Theorem \ref{Mythm}]

To prove (\ref{L1S}), notice by (\ref{isittruethough})-(\ref{isittruethoughb}) that 

\begin{align*}
L_1 = U^{(0)}(1).
\end{align*}  The representation of $U^{(0)}(1)$ is one of the fundamental relations in Appendix II which we will prove in the next section.

Notice that

\begin{align*}
\frac{1}{3}\cdot L_1 = f_1\in\mathcal{V}_{1}^{(1)}.
\end{align*}  Suppose that for some $\alpha\in\mathbb{Z}_{\ge 1}$, we have
\begin{align*}
\frac{1}{3^{2\alpha-1}}\cdot L_{2\alpha-1}\in\mathcal{V}_n^{(1)},
\end{align*} for some $n\in\mathbb{Z}_{\ge 1}$ with $n\equiv 1\pmod{3}$.  Then

\begin{align}
L_{2\alpha-1} = 3^{2\alpha-1}\cdot f_{2\alpha-1},
\end{align} with $f_{2\alpha-1}\in\mathcal{V}_n^{(1)}$.  Now,

\begin{align}
L_{2\alpha} = U_3\left(L_{2\alpha-1} \right) = U_3\left( 3^{2\alpha-1}\cdot f_{2\alpha-1} \right) = 3^{2\alpha-1}\cdot U^{(1)}\left( f_{2\alpha-1} \right).
\end{align}  By Theorem \ref{gofrom1to0}, we know that there exists some $f_{2\alpha}\in\mathcal{V}_{3n}^{(0)}$ such that

\begin{align}
U^{(1)}\left( f_{2\alpha-1} \right) = 9\cdot f_{2\alpha}.
\end{align} Therefore

\begin{align}
L_{2\alpha} = 3^{2\alpha+1}\cdot f_{2\alpha}.
\end{align}  Moreover,

\begin{align}
L_{2\alpha+1} = U_3\left( \mathcal{A}\cdot L_{2\alpha} \right) = U_3\left( \mathcal{A}\cdot 3^{2\alpha+1}\cdot f_{2\alpha}\right) = 3^{2\alpha+1}\cdot U^{(0)}\left( f_{2\alpha} \right).
\end{align}  By Theorem \ref{oddbacktoodd}, we know that there exists some $f_{2\alpha+1}\in\mathcal{V}_{9n+1}^{(1)}$ such that

\begin{align}
U^{(0)}\left( f_{2\alpha} \right) = f_{2\alpha+1}.
\end{align}  Therefore,

\begin{align}
L_{2\alpha+1} = 3^{2\alpha+1}\cdot f_{2\alpha+1}.
\end{align}

We briefly show that the power of our localizing factor for $L_{\alpha}$ matches with $\psi(\alpha)$, i.e., that

\begin{align*}
\frac{L_{2\alpha-1}}{3^{2\alpha-1}}&\in\mathcal{V}_{\psi(2\alpha-1)}^{(1)},\\
\frac{L_{2\alpha}}{3^{2\alpha+1}}&\in\mathcal{V}_{\psi(2\alpha)}^{(0)}.
\end{align*}  It is a trivial fact of modular arithmetic that for all $\alpha\ge 1$,

\begin{align*}
3^{2\alpha-1}&\equiv 3\pmod{8},\\
3^{2\alpha}&\equiv 1\pmod{8},
\end{align*} and therefore that

\begin{align*}
\left\lfloor \frac{3^{2\alpha-1}}{8} \right\rfloor &= \frac{3^{2\alpha-1}}{8} - \frac{3}{8},\\
\left\lfloor \frac{3^{2\alpha}}{12} \right\rfloor &= \frac{3^{2\alpha}}{8} - \frac{1}{8}.
\end{align*}  For example,

\begin{align*}
3\cdot\psi(2\alpha-1) &= 3\cdot \left\lfloor \frac{3^{2\alpha}}{8} \right\rfloor\\
&= 3\cdot \left(\frac{3^{2\alpha}}{8} - \frac{1}{8}\right)\\
&= \frac{3^{2\alpha+1}}{8} - \frac{3}{8}\\
&= \left\lfloor \frac{3^{2\alpha+1}}{8} \right\rfloor\\
&= \psi(2\alpha).
\end{align*}  It can similarly be proved that

\begin{align*}
3\cdot\psi(2\alpha)+1 = \psi(2\alpha+1).
\end{align*}  Therefore, our powers of $(1+9x)$ increase as in Theorem \ref{thmuyox}.  Finally, $\psi(1) = 1$ is the correct power for $L_1$.

\end{proof}

\section{Initial Relations}\label{initialrelationssection}

For $i$ fixed, we need 18 initial relations in order to properly express $U^{(i)}\left( \frac{x^{m}}{(1+9x)^{n}} \right)$ for arbitrary $m,n$.  However, notice that

\begin{align}
U^{(i)}\left( \frac{x^{m}}{(1+9x)^{n}} \right) =& \frac{1}{9^m}\cdot U^{(i)}\left( \frac{(z-1)^{m}}{z^{n}} \right)\\
=& \frac{1}{9^m}\sum_{r=0}^{m}(-1)^{m-r}{{m}\choose{r}}\cdot U^{(i)}\left( z^{r-n} \right)\\
=& \frac{1}{9^m}\sum_{r=0}^{m}(-1)^{m-r}{{m}\choose{r}}\cdot U^{(i)}\left( (1+9x)^{r-n} \right).
\end{align}  We can compute any value of $U^{(i)}\left( \frac{x^{m}}{(1+9x)^{n}} \right)$ if we have the value of $U^{(i)}\left( (1+9s)^{r} \right)$ for $-n\le r\le m-n$.

In particular, if we want expressions of $U^{(i)}\left( \frac{x^{m}}{(1+9x)^{n}} \right)$ for $1\le m,n\le 3$, we need to know $U^{(i)}\left( (1+9x)^{r} \right)$ for $-3\le r\le 2$.  However, our modular equation for $z=1+9x$ allows us to simplify even more:

\begin{align}
U^{(i)}\left( (1+9z)^n \right) &= -\sum_{k=0}^{2} b_k(\tau)\cdot U^{(i)}\left( (1+9x)^{k+n-3} \right).
\end{align}  For $n\ge 0$ we of course have

\begin{align}
U^{(i)}\left( (1+9x)^n \right) &= \sum_{k=0}^n {{n}\choose{k}}\cdot 9^k\cdot U^{(i)}\left( x^k \right),
\end{align} and we can take advantage of the modular equation for $x$.

So to derive expressions for $U^{(i)}\left( \frac{x^{m}}{(1+9x)^{n}} \right)$ for any $m,n\in\mathbb{Z}$ and $i$ fixed at 0 or 1, we need expressions for $U^{(i)}\left( x^k \right)$ for 3 consecutive values of $k$.  For both values of $i$, we need a total of 6 relations.  Actually, $U^{(1)}\left( x^0 \right)=U_3\left( 1 \right)=1$ is true trivially, so that we only really have 5 relations to prove.

Once these relations are established, 

\begin{theorem}
The relations from Theorem \ref{thmuyox}, together with the congruence conditions of Corollary \ref{congcor1}, hold for $1\le m\le 3,$ and $1\le n\le 3$.
\end{theorem}

Checking the 18 initial relations is a straightforward computation which we detail in the Mathematica supplement.  To prove the fundamental relations, we take advantage of the theory of modular functions in Section \ref{basictheorysection}.  We already proved that $x\in\mathcal{M}^0\left( \Gamma_0(6) \right)$.  After a possible normalization (i.e., multiplying to cancel a denominator of $1+9x$), we need to verify that the left-hand side of each relation is a modular function with a pole at the same cusp.  We can then to compare the principal parts and constants from either side of our prospective relation.  It is easier to expand both sides of each relation with respect to the pole at $[\infty]_6$ (i.e., the pole at $q\rightarrow 0$), so we divide by a sufficient power of $x$ to induce a pole at the cusp representable by $[\infty]_6$.

\subsection{Computing the Fundamental Relations}

The relations in Appendix II consist of equations of the form

\begin{align}
U^{(1)}\left( x^l \right) &=p_{1,l}(x)\in\mathbb{Z}[x],
\end{align} for $0\le l\le 2$, and
\begin{align}
z\cdot U^{(0)}\left( x^l \right) &=p_{0,l}(x)\in\mathbb{Z}[x]
\end{align} for $0\le l\le 2$.

We already showed that $x,z\in\mathcal{M}(\Gamma_0(6))$, and that $\mathcal{A}\in\mathcal{M}(\Gamma_0(18))$.

Notice that by (\ref{Ligoxn1a})-(\ref{Ligoxn1b}), $1/x\in\mathcal{M}^{\infty}(\Gamma_0(6))$.  Therefore, if we define $m$ as the degree of $p_{i,l}$, and then we multiply both sides of the prospective relations above by $1/x^{m}$, our relations have the form

\begin{align}
\frac{1}{x^{m}}\cdot U^{(1)}\left( x^l \right) &\in\mathbb{Z}[x^{-1}]\subseteq\mathcal{M}^{\infty}(\Gamma_0(6)),\label{initialTypeA}
\end{align}
\begin{align}
\frac{1}{x^{m}}\cdot z\cdot U^{(0)}\left( x^l \right) &\in\mathbb{Z}[x^{-1}]\subseteq\mathcal{M}^{\infty}(\Gamma_0(6)).\label{initialTypeB}
\end{align}  We need to verify that the normalized left-hand sides of each relation in Appendix II are members of $\mathcal{M}^{\infty}(\Gamma_0(6))$.  Then we compare the principal parts and constants on either side: if these are equal, then the functions must be entirely equal.

Starting with the left-hand sides, we use Theorem \ref{Ligozat} to compute Table \ref{tablew0}.

\begin{table}[hbt!]
\begin{center}
\begin{tabular}{l|c|c|c|r}
 Cusp Representative      &$\mathcal{A}(\tau)$  & $x(\tau)$ & $x(3\tau)$ & $z(3\tau)$\\
\hline
 $\infty$      & 1          & 1 & 3 & 0  \\
 $1/9$         & 4          & 0 & 0 & 0  \\
 $1/6$         & 0          & 1 & 0 & 1  \\
 $1/3$         & 0          & 0 & -1& -1 \\
 $1/2$         & -1         & 0 & 0 & 1  \\
 $2/3$         & 0          & 0 & -1 & -1  \\
$5/6$          & 0          & 1 & 0 & 1  \\
$0$             & -4         & -3 & -1 & -1  
 \end{tabular}
\caption{Order at Cusps of $\mathrm{X}_0(18)$}\label{tablew0}
\end{center}
\end{table}  Remembering Line 2 of Lemma \ref{impUprop}, we can pull everything on the left-hand side of our prospective relations inside of $U_3$:

\begin{align*}
\frac{1}{x^{m}}\cdot U^{(1)}\left( x^l \right) = U_3\left( x(3\tau)^{-m} x(\tau)^l \right),
\end{align*}
\begin{align*}
\frac{1}{x^{m}}\cdot z\cdot U^{(0)}\left( x^l \right) = U_3\left( \mathcal{A}(\tau) x(3\tau)^{-m} z(3\tau) x^l \right)
\end{align*}

Given that the largest degree in the relations of Appendix II is 11, we may simply take $m=11$ in all of our cases.  One then computes, using Table \ref{tablew0}, that

\begin{align*}
x(3\tau)^{-11} x(\tau)^l&\in\mathcal{M}^{\infty}(\Gamma_0(18)),\\
\mathcal{A}(\tau) x(3\tau)^{-11} z(3\tau) x^l&\in\mathcal{M}^{\infty}\left(\Gamma_0(18)\right).
\end{align*}  If we remember Line 3 of Lemma \ref{impUprop}, and apply it in terms of $\tau$, then

\begin{align*}
3\cdot U_3\left(f(\tau)\right) = \sum_{r=0}^2 f\left( \frac{\tau+r}{3}\right).
\end{align*}  Notice that by Theorem \ref{uelleffectsmod}, this is a member of $\mathcal{M}\left(\Gamma_0(6)\right)$.

Now, as

\begin{align*}
\tau&\rightarrow \frac{h}{k}\in\mathbb{Q},\ \mathrm{gcd}(h,k)=1,\\
\frac{\tau+r}{3}&\rightarrow \frac{h+kr}{3k}.
\end{align*}  In order for $(h+kr)/3k\in [\infty]_{18}$, (\ref{equivalentcuspsA})-(\ref{equivalentcuspsB}) require that

\begin{align*}
&y\equiv h+kr+3jk\pmod{18}\\
&18\equiv 3ky\pmod{18},\\
&\mathrm{gcd}(y,18)=1.
\end{align*}  This implies that $k=6k_1$ with $k_1\in\mathbb{Z}$, and that

\begin{align*}
&y\equiv h+kr\pmod{18},\\
&\mathrm{gcd}(y,18)=1
\end{align*} has a solution.  But this is true for all $r$, since $\mathrm{gcd}(h,k)=\mathrm{gcd}(h,6k_1)=1,$ and $\mathrm{gcd}(h+kr,18)=\mathrm{gcd}(h,18)=1$.

This means that if $f(\tau)\in\mathcal{M}^{\infty}\left( \Gamma_0(18) \right)$, then $U_3\left(f(\tau)\right)\in\mathcal{M}\left( \Gamma_0(6) \right)$ will have a pole at a cusp $[h/6k_1]_6$, with $\mathrm{gcd}(h,6k_1)=1$.  But we can verify that there exists some $y_1$ such that

\begin{align*}
&y_1\equiv h+6k_1j\pmod{6}\\
&6\equiv 6ky_1\pmod{6},\\
&\mathrm{gcd}(y_1,6)=1.
\end{align*}  By (\ref{equivalentcuspsA})-(\ref{equivalentcuspsB}), we have $[h/6k_1]_6 = [1/6]_6= [\infty]_6$, and

\begin{align*}
U_3\left( x(3\tau)^{-11} x(\tau)^l \right)&\in\mathcal{M}^{\infty}(\Gamma_0(6)),\\
U_3\left( \mathcal{A}(\tau) x(3\tau)^{-11} z(3\tau) x^l\right) &\in\mathcal{M}^{\infty}(\Gamma_0(6)),
\end{align*} for $0\le l\le 2$.  We compute the principal parts and constants of these, and compare to the principal parts and constants of the left-hand sides of $p_{i,l}(x)/x^{11}$.

Finally, we prove (\ref{zequals1plus9x}) and (\ref{modX}).  The associated computations are relatively short, and given in the Mathematica supplement.

\subsection{The Relation of $z$ to $x$}

To prove (\ref{zequals1plus9x}), we use Theorem \ref{Ligozat} to find that

\begin{align}
\mathrm{ord}_{\infty}^{(6)}(z/x) &= -1,\\
\mathrm{ord}_{1/3}^{6}(z/x) &= 0,\\
\mathrm{ord}_{1/2}^{(6)}(z/x) &= 1,\\
\mathrm{ord}_{0}^{(6)}(z/x) &= 0.
\end{align}  This, together with (\ref{Ligoxn1a})-(\ref{Ligoxn1b}), shows that

\begin{align*}
\frac{z}{x}-\frac{1}{x}\in\mathcal{M}^{\infty}(\Gamma_0(6)).
\end{align*}  A quick calculation shows that $\frac{z}{x}-\frac{1}{x}=9.$  This can be rearranged to (\ref{zequals1plus9x}).

\subsection{Proof of the Modular Equation}

To prove (\ref{modX}), we again use Table \ref{tablew0} to verify that 

\begin{align*}
x(3\tau)^{-3}\cdot x(\tau)\in\mathcal{M}^{\infty}\left(\Gamma_0(18)\right).
\end{align*}  As such, the principal part and constant of

\begin{align}
x(3\tau)^{-9}\cdot\left(x^3+\sum_{j=0}^2 a_j(3\tau) x^j\right)
\end{align} can quickly be verified to be 0, thus giving us (\ref{modX}).

\section{Appendix I}

We give the tables used in the proof of Theorems \ref{v02v1} and \ref{gofrom1to0}.

\begin{table}[hbt!]
\begin{center}
\begin{tabular}{l|c|c|c|c|c|r}
 $m$      &$r=1$  & $r=2$ & $r=3$ &$r=4$  & $r=5$ & $r=6$ \\
\hline
 1         & 0            & 0 & 1 & -1            & 0 & 1 \\
 2         & 0            & 0 & 0 & -1            & 0 & 1  \\
 3         &              & 0 & 0 & -1            & 0 & 0   \\
 4         &              & 2 & 2 & 1            & 1 & 2   \\
 5         &              & 2 & 2 & 0            & 1 & 2  \\
 6         &              &  & 2 & 0            & 1 & 2
 \end{tabular}
\caption{Value of $\displaystyle{\theta(m) + \pi_0(m,r) - \theta(r)}$ with $1\le m\le 6,$  $1\le r\le 6$}\label{tablew1}
\end{center}
\end{table}

We want each of the values above to be nonzero.  Notice that we have negative values at $r=4$ for $1\le m\le 3$.  However, we take note of the fact that $h_0(m,3,4)\equiv 0\pmod{3}$ for $1\le m\le 3$.  This provides the additional power of 3 needed.

\begin{table}[hbt!]
\begin{center}
\begin{tabular}{l|c|c|c|c|c|r}
 $m$      &$r=1$  & $r=2$ & $r=3$ &$r=4$  & $r=5$ & $r=6$ \\
\hline
 1         & -2          & -1 & 0 & -1            & 0 & 0 \\
 2         & -2          & -1 & 0 & -1            & 0 & 0  \\
 3         & -2          & -1 & 0 & -1            & 0 & 0   \\
 4         &              & 0 & 1 & 0              & 1 & 1   \\
 5         &              & 0 & 1 & 0              & 0 & 1  \\
 6         &              & 0 & 1 & -1             & 0 & 1
 \end{tabular}
\caption{Value of $\displaystyle{\theta(m) + \pi_1(m,r) -\theta(r) -2}$ with $1\le m\le 6,$  $1\le r\le 6$}\label{tablew2}
\end{center}
\end{table}

Similarly to our first table, we have some negative values to consider.  In particular, we have a value of $-1$ at $r=2$, $1\le m\le 3$.  Once again, we have $h_1(m,1,2)\equiv 0\pmod{3}$.  Similarly, $h_1(m,1,4)\equiv 0\pmod{3}$ for $1\le m\le 6$, canceling the negative values at $r=4$ for $1\le m\le 3$ and $m=6$.

This leaves only the negative value $-2$ for $r=1$, $1\le m\le 3$ that is accounted for above.

\section{Appendix II}

Here we give the six fundamental relations that are justified in Section \ref{initialrelationssection}.  The derivation of the 18 relations used in Theorem \ref{thmuyox} is given in our Mathematica supplement.

\begin{align}
U^{(1)}\left( 1 \right) &=1\\
U^{(1)}\left( x \right) &= 19 x + 360 x^2 + 1728 x^3\\
U^{(1)}\left( x^2 \right) &=10 x + 1269 x^2 + 41904 x^3 + 585792 x^4 + 3732480 x^5 + 8957952 x^6\\
U^{(0)}\left( 1 \right) &=\frac{1}{1+9x}\left( 33 x + 1392 x^2 + 21120 x^3 + 138240 x^4 + 331776 x^5 \right)\\
U^{(0)}\left( x \right) &= \frac{1}{1+9x}\big( 12 x + 2325 x^2 + 121080 x^3 + 2915136 x^4 + 37988352 x^5 + 277696512 x^6\nonumber\\ &+ 1074954240 x^7 + 1719926784 x^8 \big)\\
U^{(0)}\left( x^2 \right) &=\frac{1}{1+9x}\big( x + 1213 x^2 + 176005 x^3 + 10225152 x^4 + 318757248 x^5 + 6012278784 x^6\nonumber\\ &+ 72239910912 x^7 + 558546223104 x^8 + 
 2698565124096 x^9 + 7430083706880 x^{10}\nonumber\\ &+ 8916100448256 x^{11} \big)
\end{align}

\section{Acknowledgments}
This research was funded in whole by the Austrian Science Fund (FWF): Einzelprojekte P 33933, ``Partition Congruences by the Localization Method".  My sincerest and humblest thanks to the Austrian Government and People for their generous support.

I wish to thank Professors George Andrews and Peter Paule for their careful and dedicated work on this family of congruences.  I am also grateful to the anonymous referee for suggestions which have substantially improved the form of the paper.  A personal thanks to Marie-Sophie Haiberger for her help in the illustrations in the paper, and to my colleague Koustav Banerjee for his insistence that my ideas in this subject were worth investigating.

\end{document}

%% file: illustration_nic1_.pdf_tex
\begingroup%
  \makeatletter%
  \providecommand\color[2][]{%
    \errmessage{(Inkscape) Color is used for the text in Inkscape, but the package 'color.sty' is not loaded}%
    \renewcommand\color[2][]{}%
  }%
  \providecommand\transparent[1]{%
    \errmessage{(Inkscape) Transparency is used (non-zero) for the text in Inkscape, but the package 'transparent.sty' is not loaded}%
    \renewcommand\transparent[1]{}%
  }%
  \providecommand\rotatebox[2]{#2}%
  \newcommand*\fsize{\dimexpr\f@size pt\relax}%
  \newcommand*\lineheight[1]{\fontsize{\fsize}{#1\fsize}\selectfont}%
  \ifx\svgwidth\undefined%
    \setlength{\unitlength}{382.27962891bp}%
    \ifx\svgscale\undefined%
      \relax%
    \else%
      \setlength{\unitlength}{\unitlength * \real{\svgscale}}%
    \fi%
  \else%
    \setlength{\unitlength}{\svgwidth}%
  \fi%
  \global\let\svgwidth\undefined%
  \global\let\svgscale\undefined%
  \makeatother%
  \begin{picture}(1,0.34886689)%
    \lineheight{1}%
    \setlength\tabcolsep{0pt}%
    \put(0,0){\includegraphics[width=\unitlength,page=1]{illustration_nic1_.pdf}}%
    \put(0.11489325,0.33027898){\color[rgb]{0.1372549,0.12156863,0.1254902}\makebox(0,0)[lt]{\lineheight{1.25}\smash{\begin{tabular}[t]{l}$a_3$\end{tabular}}}}%
    \put(0.26233113,0.33074984){\color[rgb]{0.1372549,0.12156863,0.1254902}\makebox(0,0)[lt]{\lineheight{1.25}\smash{\begin{tabular}[t]{l}$a_5$\end{tabular}}}}%
    \put(0.56713419,0.32951383){\color[rgb]{0.1372549,0.12156863,0.1254902}\makebox(0,0)[lt]{\lineheight{1.25}\smash{\begin{tabular}[t]{l}$a_{2k-1}$\end{tabular}}}}%
    \put(0.7131595,0.32876831){\color[rgb]{0.1372549,0.12156863,0.1254902}\makebox(0,0)[lt]{\lineheight{1.25}\smash{\begin{tabular}[t]{l}$a_{2k+1}$\end{tabular}}}}%
    \put(0.56271988,0.0047581){\color[rgb]{0.1372549,0.12156863,0.1254902}\makebox(0,0)[lt]{\lineheight{1.25}\smash{\begin{tabular}[t]{l}$a_{2k-2}$\end{tabular}}}}%
    \put(0.73364189,0.0047581){\color[rgb]{0.1372549,0.12156863,0.1254902}\makebox(0,0)[lt]{\lineheight{1.25}\smash{\begin{tabular}[t]{l}$a_{2k}$\end{tabular}}}}%
    \put(0.84945371,0.16897036){\color[rgb]{0.1372549,0.12156863,0.1254902}\makebox(0,0)[lt]{\lineheight{1.25}\smash{\begin{tabular}[t]{l}$a_{2k+2}$\end{tabular}}}}%
    \put(0.41078921,0.32976888){\color[rgb]{0.1372549,0.12156863,0.1254902}\makebox(0,0)[lt]{\lineheight{1.25}\smash{\begin{tabular}[t]{l}$a_7$\end{tabular}}}}%
    \put(0.40576671,0.00452267){\color[rgb]{0.1372549,0.12156863,0.1254902}\makebox(0,0)[lt]{\lineheight{1.25}\smash{\begin{tabular}[t]{l}$a_6$\end{tabular}}}}%
    \put(0.26649039,0.00450305){\color[rgb]{0.1372549,0.12156863,0.1254902}\makebox(0,0)[lt]{\lineheight{1.25}\smash{\begin{tabular}[t]{l}$a_4$\end{tabular}}}}%
    \put(0.11373572,0.00414991){\color[rgb]{0.1372549,0.12156863,0.1254902}\makebox(0,0)[lt]{\lineheight{1.25}\smash{\begin{tabular}[t]{l}$a_2$\end{tabular}}}}%
    \put(-0.00178182,0.1684014){\color[rgb]{0.1372549,0.12156863,0.1254902}\makebox(0,0)[lt]{\lineheight{1.25}\smash{\begin{tabular}[t]{l}$a_1$\end{tabular}}}}%
  \end{picture}%
\endgroup%

%% file: illustration2.pdf_tex
\begingroup%
  \makeatletter%
  \providecommand\color[2][]{%
    \errmessage{(Inkscape) Color is used for the text in Inkscape, but the package 'color.sty' is not loaded}%
    \renewcommand\color[2][]{}%
  }%
  \providecommand\transparent[1]{%
    \errmessage{(Inkscape) Transparency is used (non-zero) for the text in Inkscape, but the package 'transparent.sty' is not loaded}%
    \renewcommand\transparent[1]{}%
  }%
  \providecommand\rotatebox[2]{#2}%
  \newcommand*\fsize{\dimexpr\f@size pt\relax}%
  \newcommand*\lineheight[1]{\fontsize{\fsize}{#1\fsize}\selectfont}%
  \ifx\svgwidth\undefined%
    \setlength{\unitlength}{340.14159267bp}%
    \ifx\svgscale\undefined%
      \relax%
    \else%
      \setlength{\unitlength}{\unitlength * \real{\svgscale}}%
    \fi%
  \else%
    \setlength{\unitlength}{\svgwidth}%
  \fi%
  \global\let\svgwidth\undefined%
  \global\let\svgscale\undefined%
  \makeatother%
  \begin{picture}(1,0.32091046)%
    \lineheight{1}%
    \setlength\tabcolsep{0pt}%
    \put(0,0){\includegraphics[width=\unitlength,page=1]{illustration2.pdf}}%
    \put(0.08584125,0.28433011){\color[rgb]{0.1372549,0.12156863,0.1254902}\makebox(0,0)[lt]{\lineheight{1.25}\smash{\begin{tabular}[t]{l}$a_3$\end{tabular}}}}%
    \put(0.17196715,0.15968349){\color[rgb]{0.1372549,0.12156863,0.1254902}\makebox(0,0)[lt]{\lineheight{1.25}\smash{\begin{tabular}[t]{l}$a_{2k+2}$\end{tabular}}}}%
    \put(0.18821774,0.03470612){\color[rgb]{0.1372549,0.12156863,0.1254902}\makebox(0,0)[lt]{\lineheight{1.25}\smash{\begin{tabular}[t]{l}$a_{2k}$\end{tabular}}}}%
    \put(0.08716423,0.03444152){\color[rgb]{0.1372549,0.12156863,0.1254902}\makebox(0,0)[lt]{\lineheight{1.25}\smash{\begin{tabular}[t]{l}$a_2$\end{tabular}}}}%
    \put(-0.00136508,0.16030088){\color[rgb]{0.1372549,0.12156863,0.1254902}\makebox(0,0)[lt]{\lineheight{1.25}\smash{\begin{tabular}[t]{l}$a_1$\end{tabular}}}}%
    \put(0,0){\includegraphics[width=\unitlength,page=2]{illustration2.pdf}}%
    \put(0.39149338,0.15948504){\color[rgb]{0.1372549,0.12156863,0.1254902}\makebox(0,0)[lt]{\lineheight{1.25}\smash{\begin{tabular}[t]{l}$a_{4k+3}$\end{tabular}}}}%
    \put(0,0){\includegraphics[width=\unitlength,page=3]{illustration2.pdf}}%
    \put(0.82446013,0.16063162){\color[rgb]{0.1372549,0.12156863,0.1254902}\makebox(0,0)[lt]{\lineheight{1.25}\smash{\begin{tabular}[t]{l}$a_{(2k+1)m+1}$\end{tabular}}}}%
    \put(0,0){\includegraphics[width=\unitlength,page=4]{illustration2.pdf}}%
    \put(0.17675192,0.28421986){\color[rgb]{0.1372549,0.12156863,0.1254902}\makebox(0,0)[lt]{\lineheight{1.25}\smash{\begin{tabular}[t]{l}$a_{2k+1}$\end{tabular}}}}%
    \put(0.28175231,0.28448445){\color[rgb]{0.1372549,0.12156863,0.1254902}\makebox(0,0)[lt]{\lineheight{1.25}\smash{\begin{tabular}[t]{l}$a_{2k+4}$\end{tabular}}}}%
    \put(0.2827225,0.03483842){\color[rgb]{0.1372549,0.12156863,0.1254902}\makebox(0,0)[lt]{\lineheight{1.25}\smash{\begin{tabular}[t]{l}$a_{2k+3}$\end{tabular}}}}%
    \put(0.39286046,0.03475022){\color[rgb]{0.1372549,0.12156863,0.1254902}\makebox(0,0)[lt]{\lineheight{1.25}\smash{\begin{tabular}[t]{l}$a_{4k+1}$\end{tabular}}}}%
    \put(0.39583716,0.28415371){\color[rgb]{0.1372549,0.12156863,0.1254902}\makebox(0,0)[lt]{\lineheight{1.25}\smash{\begin{tabular}[t]{l}$a_{4k+2}$\end{tabular}}}}%
    \put(0.69663836,0.28419781){\color[rgb]{0.1372549,0.12156863,0.1254902}\makebox(0,0)[lt]{\lineheight{1.25}\smash{\begin{tabular}[t]{l}$a_{(2k+1)m}$\end{tabular}}}}%
    \put(0.69699115,0.03503686){\color[rgb]{0.1372549,0.12156863,0.1254902}\makebox(0,0)[lt]{\lineheight{1.25}\smash{\begin{tabular}[t]{l}$a_{(2k+1)m-1}$\end{tabular}}}}%
    \put(0,0){\includegraphics[width=\unitlength,page=5]{illustration2.pdf}}%
  \end{picture}%
\endgroup%